\documentclass[12pt]{article}
\usepackage[margin=0.7in]{geometry}
\usepackage{amssymb,amsmath,amsthm}
\usepackage{graphicx,epstopdf,color}
\usepackage{esint}

\allowdisplaybreaks
\usepackage[alphabetic]{amsrefs}

\usepackage[latin1]{inputenc}
\newcommand{\R}{\mathbb{R}}

\newcommand{\N}{\mathbb{N}}

\usepackage{miama}
\usepackage[T1]{fontenc}

\numberwithin{equation}{section}

\newtheorem{thm}{Theorem}[section]

\newtheorem{lem}[thm]{Lemma}
\newtheorem{prop}[thm]{Proposition}

\newtheorem{rem}[thm]{Remark}
\newtheorem{ex}[thm]{Example}

\renewcommand{\leq}{\leqslant}
\renewcommand{\le}{\leqslant}
\renewcommand{\geq}{\geqslant}
\renewcommand{\ge}{\geqslant}

\usepackage{authblk}

\begin{document}

\title{The Fractional Malmheden Theorem}

\author{Serena Dipierro, Giovanni Giacomin and Enrico Valdinoci
\thanks{Department of Mathematics and Statistics,
University of Western Australia, 35 Stirling Highway,
WA6009 Crawley, Australia.\\
{\tt serena.dipierro@uwa.edu.au,
giovanni.giacomin@research.uwa.edu.au,
enrico.valdinoci@uwa.edu.au\\
SD and EV are members of INdAM and AustMS.
Supported by the Australian Laureate Fellowship FL190100081 ``Minimal
surfaces, free boundaries and partial differential equations''
and
by the Australian Research Council DECRA DE180100957
``PDEs, free boundaries
and applications''.}}}

\maketitle

\begin{abstract}
We provide a fractional counterpart of the classical results by Schwarz and Malmheden on harmonic functions. From that we obtain a representation formula for $s$-harmonic functions as a linear superposition of weighted classical harmonic functions which also entails a new proof of the fractional Harnack inequality. This proof also leads to optimal constants for the fractional Harnack inequality in the ball.  
\end{abstract}

\bigskip
\begin{center}
{\fmmfamily{ \Huge To Neil, the Master of us all.}}
 \end{center}\bigskip
 
\section{Introduction}

In 1934, Harry William Malmheden \cite{6} proved a simple algorithm to compute the value of a harmonic function
at a point of $B_1$, knowing its value on the boundary. 

The Malmheden Theorem makes use of two fundamental geometric ingredients:
\begin{enumerate}
\item the notion of affine interpolation between the values of a given function at two different points of the space,
\item the projections of a point inside a ball to the boundary in a given direction.
\end{enumerate}
Hence, to state the Malmheden Theorem explicitly, we now formalize these two notions into a precise mathematical setting.
We start by introducing a notation
for the affine interpolation between the values of some given function. That is, given a set~$K\subseteq\R^n$,
a function~$f:K\to\R$, two distinct points~$a$, $b\in K$, and a point~$x$ on the segment~$L$ joining~$a$ and~$b$, we define~$
\mathcal{L}_{f}^{a,b}(x)$ as the affine function on~$L$ such that~$\mathcal{L}_{f}^{a,b}(a)=f(a)$ and~$
\mathcal{L}_{f}^{a,b}(b)=f(b)$. 

Of course, one can write this affine function explicitly by using the analytic expression
\begin{equation}\label{FORMULR}
\mathcal{L}_{f}^{a,b}(x)=\frac{(x-a)\cdot e}{|b-a|}\,f(b)+\frac{(b-x)\cdot e}{|b-a|}\,f(a),\qquad{\mbox{where }}e:=\frac{b-a}{|b-a|}.
\end{equation}
One can call~$\mathcal{L}_{f}^{a,b}(x)$ the ``affine function of~$f$ with extrema~$a$ and~$b$
evaluated at the point~$x$''.

\begin{figure}[h]
		\centering
		\includegraphics[height=.2\textheight]{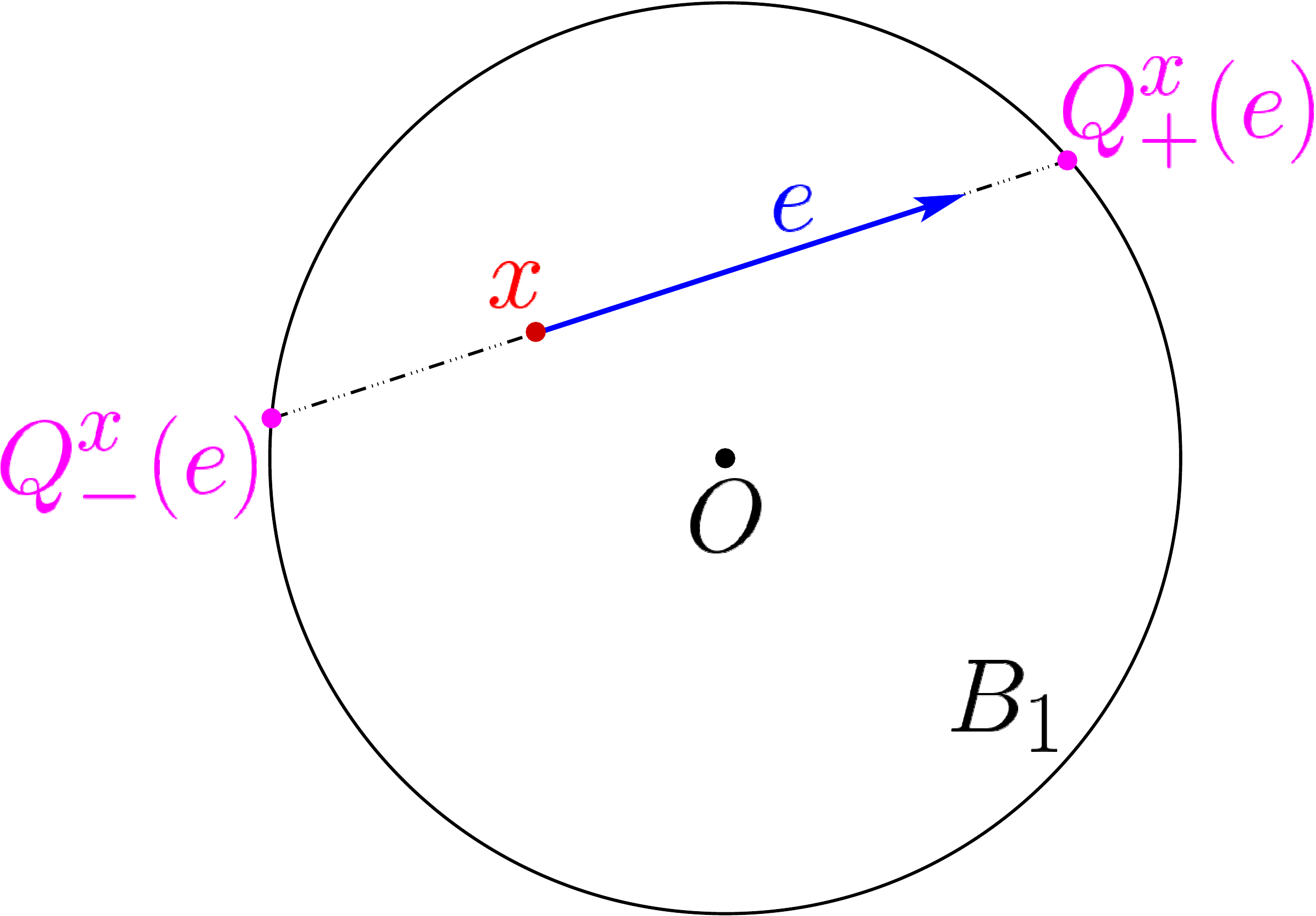}
	\caption{The projections~$Q_\pm^x(e)$ used in the Malmheden Theorem.}	\label{DIFI}
\end{figure}

Now we discuss the notation regarding the projections of a point
inside a ball to the boundary of the ball in a given direction.
For this, given a point $x\in B_1$
and a direction~$e\in \partial B_1$ we consider the intersections $Q_{+}^x(e)$ and~$Q_{-}^x(e)$ of $\partial B_1$ with the straight line passing through $x\in B_1$ with direction $e$, with the convention that~$Q_{+}^x(e)-Q_{-}^x(e)$
has the same orientation of~$e$, see Figure~\ref{DIFI}.

Clearly, from the analytic point of view, one can write explicitly these projections in the form
\begin{equation}\label{12}
\begin{split}
&Q_{+}^x(e):=x+r_{+}^x(e)e\\ {\mbox{and }}\qquad
&Q_{-}^x(e):=x+r_{-}^x(e)e,
\end{split}
\end{equation}
where \begin{equation}\label{23}
\begin{split}
&r_{+}^x(e)=-x\cdot e+\sqrt{(x\cdot e)^2-\vert x\vert^2+1}\\ {\mbox{and }}\qquad
&r_{-}^x(e)=-x\cdot e-\sqrt{(x\cdot e)^2-\vert x\vert^2+1}.
\end{split}
\end{equation}

We note from equations \eqref{12} and \eqref{23} that $Q_{\pm}^x(e)$ are
continuous functions in~$(x,e)\in B_1\times \partial B_1$. Moreover,
\begin{equation}\label{mandorla}
\lim_{x\to 0}Q_{\pm}^x(e)=\pm e
\end{equation}
for each $e\in \partial B_1$. This tells us that the maps $Q_{\pm}^x$ simply reduce to $\pm id_{\partial B_1}$ when $x=0$. 
\medskip

Given a boundary datum~$f:\partial B_1\to\R$,
the core of the Malmheden Theorem is thus to consider, for every point~$x\in B_1$ and every direction~$e\in\partial B_1$,
the affine function of~$f$ with extrema~$Q_{-}^x(e)$ and~$Q_{+}^x(e)$, namely the function
\begin{equation}\label{LINE-s1}
\mathcal{L}_{f}^{Q_{-}^x(e),Q_{+}^x(e)}(x)\end{equation} and then to average in all directions~$e$.

The remarkable result by Malmheden is that this averaging procedure of linear interpolations
produces precisely the solution of the classical Dirichlet problem in~$B_1$ with boundary datum~$f$, according to the
following classical statement (see~\cite{6}):

\begin{thm}[Malmheden Theorem]\label{MahT}
Let~$n\ge2$ and~$f:\partial B_1\rightarrow\mathbb{R}$ be continuous. Then 
\begin{equation}\label{MAHMEQ}
u_f(x):=\fint_{\partial B_1}\mathcal{L}_{f}^{Q_{-}^x(e),Q_{+}^x(e)}(x)\,dH_e^{n-1}
\end{equation}
is the harmonic function in $B_1$ with boundary datum $f$.
\end{thm}

As usual, here above and in the rest of this paper, we denoted by~$H^{n-1}$ the $(n-1)$-Hausdorff measure
(hence, the integral on the right hand side of~\eqref{MAHMEQ} is simply the spherical integral along~$\partial B_1$;
we kept the explicit notation with the Hausdorff measure to have a typographical evidence
of the surface integrals, to be distinguished by the classical volume ones).

We remark that Theorem~\ref{MahT} contains the Mean Value Theorem for harmonic functions as a particular case:
indeed, in light of~\eqref{FORMULR} and~\eqref{mandorla},
if we take~$x:=0$ then~\eqref{MAHMEQ} reduces to
\begin{equation}\label{djhiwfuiwehjk0987654rtyujhgdfg}
u_f(0)=\fint_{\partial B_1}\mathcal{L}_{f}^{-e,e}(0)\,dH_e^{n-1}
=\fint_{\partial B_1}\left(\frac{f(e)}2+\frac{f(-e)}2\right)\,dH_e^{n-1}=
\fint_{\partial B_1}f(e)\,dH_e^{n-1},\end{equation}
which is the content of the Mean Value Theorem.

We also stress that an elegant result such as Theorem~\ref{MahT} is specific for balls and
cannot be extended in general to other domains, as pointed out in~\cite{5}. 

Interestingly, Theorem~\ref{MahT} contains as a particular case a classical result due to Hermann Amandus Schwarz~\cite{MR0392470}
about the Dirichlet problem in the plane and related to conformal mappings in the complex framework.

\begin{figure}[h]
		\centering
		\includegraphics[height=.2\textheight]{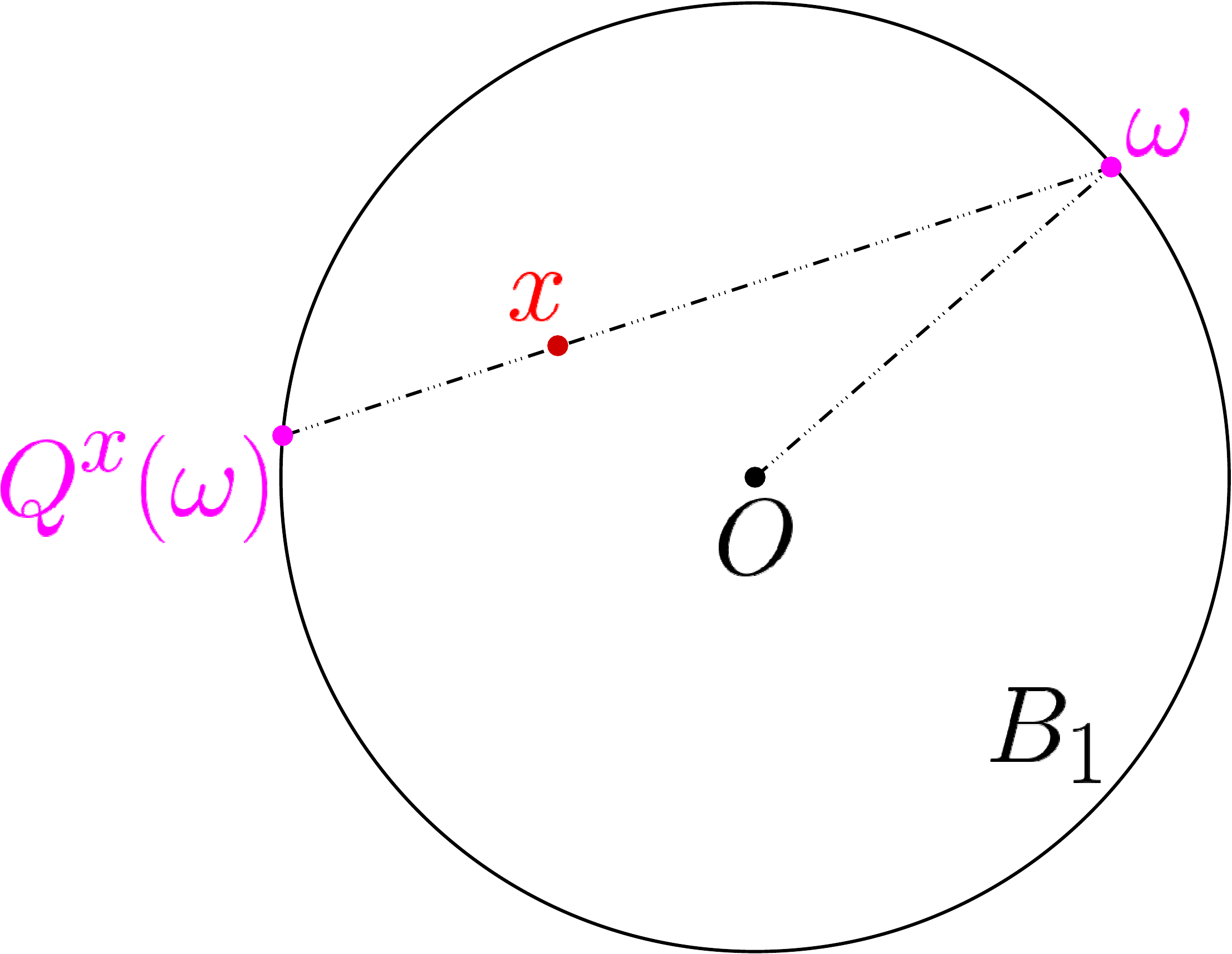}
	\caption{The reflection~$Q^x(\omega)$ used in the Schwarz Theorem.} \label{DIFI2}
\end{figure}

To state Schwarz result it is convenient to introduce the reflection of a point~$\omega\in\partial B_1$ through a point~$x\in B_1$, see Figure~\ref{DIFI2}. 
More precisely, given~$x\in B_1$ and $\omega\in \partial B_1$ we define 
\begin{equation}\label{REFKE}
Q^x(\omega):=\omega-2\frac{(x-\omega)\cdot \omega}{\vert x-\omega\vert^2}(x-\omega).
\end{equation}
Comparing with~\eqref{12}, one sees that if~$e:=\frac{\omega-x}{|\omega-x|}$ then~$Q_+^x(e)=\omega$
and~$Q_-^x(e)=Q^x(\omega)$.

In this setting, the result by Schwarz is that the average of the boundary datum composed with the above reflection returns
the solution of the Dirichlet problem in the ball. More explicitly:

\begin{thm}[Schwarz Theorem]\label{SCHW}
Let $n=2$ and~$f:\partial B_1\rightarrow\mathbb{R}$ be continuous. Then
\begin{equation}\label{scc}
u_f(x):=\fint_{\partial B_1} f(Q^{x}(e))\,dH_e^{n-1}
\end{equation}
is the harmonic function in $B_1$ with external datum $f$.
\end{thm}

Theorem~\ref{SCHW} can be proved in several ways using
either complex or real analysis (see e.g.~\cite{zbMATH02701995, MR86885, MR1272823}), but it is
also a direct consequence of Theorem~\ref{MahT}, 
see e.g.~\cite{2} for a detailed presentation of this classical
argument.

\begin{ex}\label{ESES-1} {\rm
A very neat application of Theorem~\ref{SCHW} (see e.g.~\cite{MR1446490})
consists in the determination of the stationary temperature~$u$ at a point~$x$ in a plate (say~$B_1$)
when the temperature along the boundary of the plate is~$1$ along some arc~$\Sigma$
and~$0$ outside. In this case, the reflection in~\eqref{REFKE} sends~$\Sigma$ into an arc~$\Sigma'$
(the symmetric of~$\Sigma$ through~$x$, see Figure~\ref{DIFI3}) and it therefore follows from
Theorem~\ref{SCHW} that
$$ u(x)=\frac{|\Sigma'|}{2\pi},$$
where~$|\Sigma'|$ is the length of the arc~$\Sigma'$, thus providing an elementary geometric construction
to solve a problem of physical relevance.}\end{ex}\medskip

\begin{figure}[h]
		\centering
		\includegraphics[height=.2\textheight]{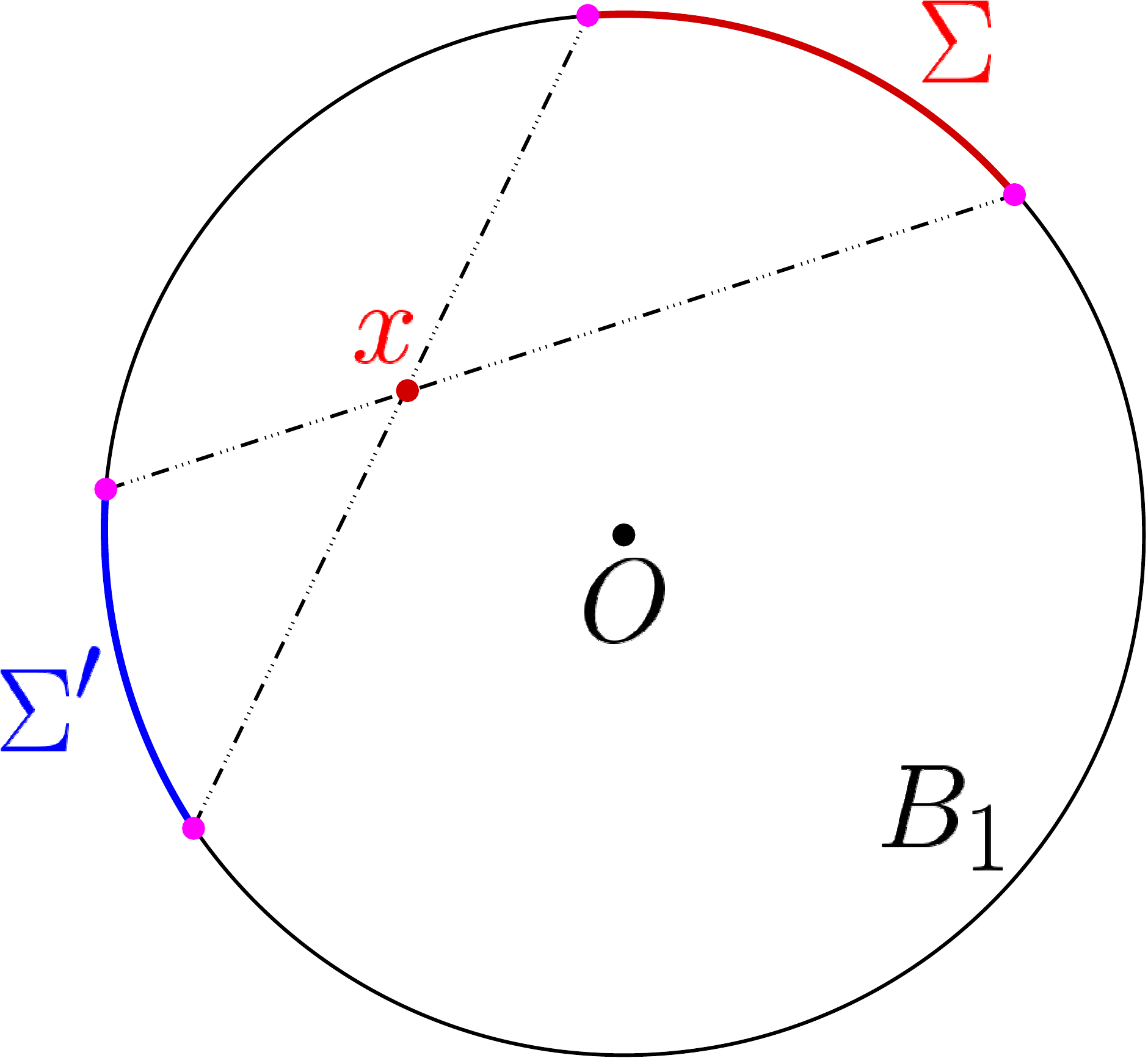}
	\caption{The geometric construction to detect the temperature of a plate at the point~$x$.}
			\label{DIFI3}
\end{figure}

The objective of this paper
is to obtain a fractional counterpart for the Malmheden and
Schwarz Theorems.

We will thus replace the notion of harmonic functions in~$B_1$ with that of $s$-harmonic functions,
namely functions whose fractional Laplacian vanishes in~$B_1$, that is, for all~$x\in B_1$,
\begin{equation}\label{LAPLAFRA} \int_{\R^n}\frac{u(x)-u(y)}{|x-y|^{n+2s}}\,dy=0,\end{equation}
where the integral above is intended in the principal value sense. Here above and throughout the paper
the fractional parameter~$s\in(0,1)$.

Rather than a boundary value along~$\partial B_1$, as usual in the nonlocal setting,
we complement~\eqref{LAPLAFRA}
with an external condition of the type~$u=f$ in~$\R^n\setminus B_1$.

We recall that in general $s$-harmonic functions behave way more wildly that their classical counterparts,
see e.g.~\cite{MR3626547}. Therefore, in principle one cannot easily expect that a ``simple formulation''
such as the one in Theorems~\ref{MahT} and~\ref{SCHW} accounts for all the complex situations arising in the fractional
setting.

However, we will prove that a counterpart of Theorems~\ref{MahT} and~\ref{SCHW} carries over to the case of the fractional Laplacian,
considering the following structural modifications:
\begin{enumerate}
\item the classical spherical averages are replaced by suitable weighted averages on spheres of radius larger than~$1$,
\item the geometric transformations in~\eqref{12} and~\eqref{REFKE} are scaled in dependence of the radius of each of these spheres.
\end{enumerate}

To clarify these points, and thus reconsider~\eqref{LINE-s1} in a nonlocal setting, given~$\rho>1$ and~$f:\R^n\setminus B_1\to\R$,
for all~$x\in\partial B_1$ we define
\begin{equation}\label{1.9BIS}
f_\rho(x):=f(\rho x).\end{equation}
Hence, in the notation of~\eqref{FORMULR}, we define
\begin{equation}\label{NEMDKJA} \mathcal{L}_{f,e,\rho}(x):=
\mathcal{L}_{f_\rho}^{Q^{x/\rho}_-(e),Q^{x/\rho}_+(e)}\left(\frac{x}\rho\right).\end{equation}
Notice that when~$\rho=1$ the above setting reduces to~\eqref{LINE-s1},
otherwise one is considering here a similar framework but for a rescaled version of the function~$f$ and rescaled points.

To detect the long-range effect of the fractional Laplacian, it is also useful to consider the kernel
\begin{equation}\label{Edefin999}
B_1\times(1,+\infty)\ni(x,\rho)\longmapsto {\mathcal{E}}(x,\rho):=
c(n,s)\,\frac{ \rho\,(1-\vert x\vert^2)^s}{(\rho^2-1)^s
\,(\rho^2-\vert x\vert^2)},\end{equation}
where
\begin{equation}\label{CNS}
c(n,s):=\frac{\displaystyle\Gamma\left(\frac{n}2\right)\,\sin(\pi s)}{\pi^{\frac{n}2+1}}.\end{equation}
For our purposes, the kernel~${\mathcal{E}}$ will play the role of a suitable spherical average\footnote{More precisely, the intuition behind~${\mathcal{E}}$
is \label{INTUK} that it satisfies, for each~$x\in B_1$ and~$\rho\in(1,+\infty)$,
$$ {\mathcal{E}}(x,\rho)=\rho^{n-1}\fint_{\partial B_\rho} P(x,y)\,dH^{n-1}_y,$$
where~$P$ is the Fractional Poisson Kernel.

This relation can be obtained as a consequence of
our Fractional Malmheden Theorem.
Though our approach does not pass explicitly through this
identity, for the sake of completeness we provide an independent proof in Appendix~\ref{SELF}.}
of a Fractional Poisson Kernel
and the constant~$c(n,s)$ is merely needed for normalization purposes.

We also define the space
\begin{equation}\label{rightspace22}
L_s^1(\mathbb{R}^n\setminus B_1):=\left\lbrace f:\R^n\to\R {\mbox{ measurable }}\, :\, \int_{\mathbb{R}^n\setminus B_1}\frac{| f(x)|}{ | x|^{n+2s}}\,dx<\infty \right\rbrace.
\end{equation}

With this, we can state the main result of this paper as follows:

\begin{thm}[Fractional Malmheden Theorem]\label{Ted}
Let $n\geq 2$, $s\in (0,1)$, $R>1$
and~$f\in L^\infty(B_{R}\setminus B_1)\cap L_s^1(\mathbb{R}^n\setminus B_1)$. 

Then, the unique solution (up to a zero measure subset of $\mathbb{R}^n\setminus B_1$) to the problem 
\begin{equation}\label{FRDIPR}
\begin{cases}
\begin{split}
(-\Delta)^s u&=0\;\;\textrm{in}\;\; B_1,\\
u&=f\;\;\textrm{in}\;\;\mathbb{R}^n\setminus B_1
\end{split}
\end{cases}
\end{equation}
can be written, for each $x\in B_1$, as
\begin{equation}\label{141021}
u_f^{(s)}(x):=\int_{1}^\infty \left(\int_{\partial B_1}
{\mathcal{E}}(x,\rho)\,
\mathcal{L}_{f,e,\rho}(x)\,dH_{e}^{n-1}\right)\,d\rho.
\end{equation} 
\end{thm}

As a fractional counterpart of the observation in~\eqref{djhiwfuiwehjk0987654rtyujhgdfg},
we point out that Theorem~\ref{Ted} entails as a straightforward
consequence the Mean Value Formula
for~$s$-harmonis functions. Indeed, by the changes of
variable~$e:=\omega/|\omega|$ and~$y:=\rho\omega/|\omega|$,
\begin{eqnarray*}
u_f^{(s)}(0)&=&\int_{1}^\infty \left(\int_{\partial B_1}
{\mathcal{E}}(0,\rho)\,
\mathcal{L}_{f,e,\rho}(0)\,dH_{e}^{n-1}\right)\,d\rho\\&=&
c(n,s)\int_{1}^\infty \left(\int_{\partial B_1}
\frac{ 1}{\rho(\rho^2-1)^s}\,
\mathcal{L}_{f_\rho}^{-e,e}(0)\,dH_{e}^{n-1}\right)\,d\rho
\\&=&
c(n,s)\int_{1}^\infty \left(\int_{\partial B_1}
\frac{ 1}{\rho(\rho^2-1)^s}\,
\left(\frac{f(\rho e)}2+\frac{f(-\rho e)}2
\right)\,dH_{e}^{n-1}\right)\,d\rho\\&=&
c(n,s)\int_{1}^\infty \left(\int_{\partial B_\rho}
\frac{ 1}{\rho^n (\rho^2-1)^s}\,
\left(\frac{f(\omega)}2+\frac{f(-\omega)}2
\right)\,dH_{\omega}^{n-1}\right)\,d\rho\\&=&
c(n,s)\int_{\R^n\setminus B_1}
\frac{ 1}{|y|^n (|y|^2-1)^s}\,
\left(\frac{f(y)}2+\frac{f(-y)}2
\right)\,dy\\
&=&c(n,s)\int_{\R^n\setminus B_1}
\frac{ f(y)}{|y|^n (|y|^2-1)^s}\,dy,
\end{eqnarray*}
which is the Mean Value Formula
for~$s$-harmonis functions, see e.g. formula~(1.3) in~\cite{MR4149297}.
\medskip

We consider Theorem~\ref{Ted} as the natural fractional counterpart of
Theorem~\ref{MahT} and we mention that indeed one can ``recover'' Theorem~\ref{MahT} in the limit as~$s\nearrow1$, according to
the following result:

\begin{prop}\label{sto1}
Let $n\geq 2$, $s_0\in (0,1)$, $R>1$
and~$f\in  C(B_{R}\setminus B_1)\cap L_s^1(\mathbb{R}^n\setminus B_1)$ 
for each $s\in (s_0,1]$. Then, for each $x\in B_1$, it holds that 
\begin{equation*}
\lim_{s\nearrow 1} u_f^{(s)}(x)= u_f(x),
\end{equation*}
where $u_f^{(s)}$ and $ u_f$ are defined in~\eqref{141021} and~\eqref{MAHMEQ}, respectively.
\end{prop}

As a straightforward consequence of the classical Malmheden Theorem (Theorem~\ref{MahT}) and its fractional formulation Theorem (Theorem~\ref{Ted}), we deduce the following result. 

\begin{thm}[An $s$-harmonic function is the superposition of classical harmonic functions]\label{sharashar}
Let $n\geq 2$, $s\in (0,1)$ and $f\in C(\mathbb{R}^n\setminus B_1)\cap L_s^1(\mathbb{R}^n\setminus B_1)$. For each $\rho>1$ we define $ {u}_{f_\rho}$ as the unique solution to the Dirichlet problem
\begin{equation}\label{kurt}
\begin{cases}
\begin{split}
\Delta u&=0\;\;\textrm{in}\;\;B_1,\\
u|_{\partial B_1}&=f_\rho|_{\partial B_1},
\end{split}
\end{cases}
\end{equation}
where~$f_\rho$ is defined in~\eqref{1.9BIS}.

Then the unique solution $u_f^{(s)}$ to \eqref{FRDIPR} can be written as
\begin{equation}\label{WEIGHT}
u_f^{(s)}(x)=\vert \partial B_1\vert \int_{1}^\infty \mathcal{E}(x,\rho)\, {u}_{f_\rho}\left(\frac{x}{\rho}\right)\,d\rho.
\end{equation}
\end{thm}

The interest of Theorem \ref{sharashar} is that it allows us to write an $s$-harmonic function in $B_1$ as a weighted integral of classical harmonic functions, where the weight coincide with $\mathcal{E}(x,\rho)$.
Besides being interesting in itself, this result is very useful
to deduce properties of $s$-harmonic functions, as the Harnack inequality (see Section  \ref{Har}),
starting from their local counterpart. 

As a matter of fact, as a consequence of Theorem \ref{sharashar} one obtains a new proof of the Harnack inequality for $s$-harmonic functions in $B_1$. The result goes as follows: 

\begin{thm}[Harnack inequality]\label{JHRGNA:THM}
Let $n\geq 2$, $s\in (0,1)$, $R>1$ and $u$ be non negative, $s$-harmonic in $B_1$ and such that $u \in
L^\infty(B_{R}\setminus B_1)\cap L_s^1(\mathbb{R}^n \setminus B_1) $. 

Then, for each $r\in(0,1)$ and $x\in B_r$,
\begin{equation}\label{FracHarn}
\frac{(1-r^2)^s}{(1+r)^n}u(0)\leq u(x)\leq \frac{(1-r^2)^s}{(1-r)^n}u(0).
\end{equation}
The constants in \eqref{FracHarn} are optimal, and for $s\nearrow 1$ they converge to the optimal constants of the classical Harnack inequality in $B_r$ for harmonic functions in $B_1$.
\end{thm} 

For different proofs of the fractional Harnack inequality see~\cite{MR1941020, 8, MR2817382} and the references therein.
\medskip

Another consequence of Theorem~\ref{Ted} is the fractional version
of Schwarz result:

\begin{thm}[Fractional Schwarz Theorem]\label{N=2}
Let $n=2$, $s\in (0,1)$, $R>1$
and $f\in L^\infty(B_R\setminus B_1)\cap L_s^1(\mathbb{R}^2\setminus B_1)$.
Then, the unique solution (up to a zero measure subset of $\mathbb{R}^2\setminus B_1$) to the problem 
\begin{equation*}
\begin{cases}
\begin{split}
(-\Delta)^s u&=0\;\;\textrm{in}\;\; B_1,\\
u&=f\;\;\textrm{in}\;\;\mathbb{R}^2\setminus B_1
\end{split}
\end{cases}
\end{equation*}
can be written, for each $x\in B_1$, as
\begin{equation}\label{SBCLM}
u_f^{(s)}(x):=\int_{1}^\infty \left(\int_{\partial B_1} {\mathcal{E}}(x,\rho)\,f_\rho( Q^{x/\rho}(e))\,dH_e^{1}\right)\,d\rho.
\end{equation}
\end{thm}

This is a fractional counterpart of Theorem~\ref{SCHW}, in the sense of Proposition~\ref{Monday} below.
Proposition~\ref{Monday} is a straightforward consequence of Theorems~\ref{SCHW}
and~\ref{N=2} and Proposition~\ref{sto1}.

\begin{prop}\label{Monday}
Let $n=2$, $s_0\in (0,1)$, $R>1$
and~$f\in C(B_R\setminus B_1)\cap
L_s^1(\mathbb{R}^2\setminus B_1)$ for each $s\in (s_0,1]$. Then, for each $x\in B_1$, it holds that
\begin{equation*}
\lim_{s\nearrow 1} u_f^{(s)}(x)=u_f(x)
\end{equation*} 
where $u_f^{(s)}$, $u_f$ are defined in \eqref{SBCLM} and \eqref{scc}, respectively. 
\end{prop}

\begin{rem}\label{RemonNormE}{\rm
It is worth pointing out that from Theorem \ref{Ted} we can evince the identity
\begin{equation}\label{NormE}
\int_{1}^\infty\mathcal{E}(x,\rho)\,d\rho=\frac{1}{\vert \partial B_1\vert},
\end{equation}
for each $x\in B_1$. Indeed, if we consider as external data $f=1$ in $\mathbb{R}^n\setminus B_1$, then the unique solution to the problem \eqref{FRDIPR} is $u=1$ in $\mathbb{R}^n$.
Therefore, according to~\eqref{141021}
and the fact that in this case the linear interpolation $\mathcal{L}_{1,e,\rho}(x)=1$ for each $x\in B_1$,
we obtain that
\begin{equation*}
1=\int_{1}^\infty\left(\int_{\partial B_1}\mathcal{E}(x,\rho)\mathcal{L}_{1,e,\rho}(x)\,dH_e^{n-1}\right)\,d\rho
=\int_{1}^\infty \vert \partial B_1\vert\, \mathcal{E}(x,\rho)\,d\rho,
\end{equation*}
which gives \eqref{NormE}.}
\end{rem}

As an application of Theorem~\ref{N=2}, we have:

\begin{ex}\label{ESES-2} {\rm
Let~$n=2$ and take an arc $\Sigma\subset\partial B_1$. Consider the function defined
on~$\mathbb{R}^2\setminus B_1$ as
\begin{equation}\label{temp}
\tilde{\chi}_{\Sigma}(y):=\begin{cases}
\begin{split}
1\qquad \textrm{if}\;\; & \displaystyle \frac{y}{\vert y\vert} \in \Sigma,\\
0\qquad \textrm{if}\;\; & \displaystyle \frac{y}{\vert y\vert} \in \partial B_1\setminus \Sigma.
\end{split}
\end{cases}
\end{equation} 
It is clear that $\tilde{\chi}_{\Sigma}$ is positively homogeneous of degree zero, and furthermore $\tilde{\chi}_{\Sigma}\in L^\infty (\mathbb{R}^2\setminus B_1)\subset L_s^1(\mathbb{R}^2\setminus B_1)$.  Then by Theorem \ref{N=2} we get that for each $x\in B_1$
\begin{equation}\label{9102021}
u_{\tilde{\chi}_{\Sigma}}^{(s)}(x)=\int_{1}^\infty \mathcal{E}(x,\rho)\vert \Sigma_{x/\rho}'\vert\,d\rho 
\end{equation} 
where $\Sigma_{x/\rho}'$ is the projected arc of $\Sigma$ on $\partial B_1$ through the focal point $x/\rho$, as constructed in Example~\ref{ESES-1}. We denoted with $\vert \Sigma_{x/\rho}'\vert$ its length. 

This gives a simple geometrical procedure to compute the 
solution of
\begin{equation*}
\begin{cases}
\begin{split}
(-\Delta)^s u&=0\;\;\textrm{in}\;\; B_1,\\
u&=\tilde{\chi}_{\Sigma}\;\;\textrm{in}\;\;\mathbb{R}^2\setminus B_1
\end{split}
\end{cases}
\end{equation*}
at a point $x$ of the two dimensional disc when the non local boundary condition is given by \eqref{temp}. Note that as $\rho$ is getting larger, the measure of $\Sigma_{x/\rho}'$ reaches the one of $\Sigma$, or more precisely
\begin{equation*}
\lim_{\rho\to \infty}\vert \Sigma_{x/\rho}'\vert =\vert \Sigma\vert.
\end{equation*}  
If $x=0$, formula \eqref{9102021} boils down to 
\begin{equation}\label{FracTemp}
u_{\tilde{\chi}_{\Sigma}}^{(s)}(0)=c(n,s)\vert \Sigma\vert \int_{1}^\infty \frac{1}{\rho(\rho^2-1)^s}\,d\rho
=\frac{\vert\Sigma\vert}{2\pi},
\end{equation}
where we have applied identity \eqref{NormE}. This example can be seen as the fractional counterpart of Example~\ref{ESES-1}.}\end{ex}
\bigskip

This paper is organized as follows. In Section \ref{Prel} we give some preliminary results on the $s$-harmonic function written as a convolution with the Fractional Poisson Kernel.

Section \ref{TTTTT} is devoted to the proofs of the fractional
Malmheden and Schwarz results, that is Theorems~\ref{Ted} and~\ref{N=2}, and of
the convergence result in Proposition~\ref{sto1}.

In Section \ref{Har} we use these results to provide a proof of the well-known Harnack inequality for $s$-harmonic functions under some regularity assumptions on the external datum $f:\mathbb{R}^n\setminus B_1\rightarrow \mathbb{R}$, that is we prove Theorem~\ref{JHRGNA:THM}.

\section{Preliminary results on the Fractional Poisson Kernel}\label{Prel}

In this section, we revisit the well-established result according to which fractional harmonic functions
can be represented as an integral of the datum outside the domain against a suitable Poisson Kernel.
For completeness, we extend this result to the case in which the datum is not necessarily continuous,
so to be able to present the results of this paper in a suitable generality.
Notice that the extension to functions that are not necessarily continuous is also useful for us to
comprise situations as in Example~\ref{ESES-2}.

The framework that we consider is the following.
For $n\ge2$ and $s\in (0,1)$, we consider the space~$ L_s^1(\mathbb{R}^n\setminus B_1)$
as defined in~\eqref{rightspace22}. Given~$f\in L_s^1(\mathbb{R}^n\setminus B_1)$, we denote
the norm on~$L_s^1(\mathbb{R}^n\setminus B_1)$ by
$$ \|f\|_{L_s^1(\mathbb{R}^n\setminus B_1)}:=\int_{\R^n\setminus B_1}\frac{|f(x)|}{|x|^{n+2s}}\,dx.$$

Furthermore, we define the following \textit{Fractional Poisson Kernel} in the unit ball
\begin{equation}\label{PoKe}
P(x,y):=c(n,s)\left(\frac{1-\vert x\vert^2}{\vert y\vert^2-1}\right)^s\frac{1}{\vert x-y\vert^n}
\end{equation}
for $x\in B_1$ and $y\in \mathbb{R}^n\setminus B_1$, and $c(n,s)$ is the normalizing constant in~\eqref{CNS}.

As customary, the role of the constant~$c(n,s)$ is to normalize the Poisson Kernel, namely we have that
\begin{equation}\label{TRE} \int_{\R^n\setminus B_1} P(x,y)\,dy=1,\end{equation}
see e.g. formula (1.14) and Lemma~A.5 in \cite{1}.

We also remark that
$$ P(\cdot,y)\in C^\infty (B_1)$$
and, for every~$\rho\in(0,1)$, $\alpha\in\N^n$ and~$y\in\R^n\setminus B_1$,
\begin{equation}\label{BIS}
\sup_{x\in B_\rho}|D_x^\alpha P(x,y)|\le \frac{C_\rho}{(|y|-1)^s\,|y|^{n+s+|\alpha|}},
\end{equation}
where~$C_\rho>0$ depends only on~$\rho$, $n$ and~$s$ and, as usual, we have denoted the length
of the multi-index~$\alpha$ as~$ |\alpha|:=\alpha_1+\dots+\alpha_n$.

Then, we define
\begin{equation}\label{Harms}
u_f^{(s)}(x):=\begin{cases}
\begin{split}\displaystyle
\int_{\mathbb{R}^n\setminus B_1}P(x,y)f(y)\,dy \;\;&\textrm{if}\;x\in B_1,\\
f(x)\;\;\;\;\;\;\;\;\;\;\;\;\;\;\;\;&\textrm{if}\; x\in \mathbb{R}^n\setminus B_1,
\end{split}
\end{cases}
\end{equation}
and we have the following result on the representation of $s$-harmonic functions:

\begin{thm}\label{Thm}
Let $n\ge2$, $s\in (0,1)$ and $f\in C(\mathbb{R}^n\setminus B_1)\cap L_s^1(\mathbb{R}^n\setminus B_1)$. Then the function in~\eqref{Harms} is the unique pointwise continuous solution to the problem
\begin{equation}\label{oewrui35tu73485435u75465jhgdfjkgjd}
\begin{cases}
\begin{split}
(-\Delta)^s u_f^{(s)}&=0\;\;\textrm{in}\;\; B_1, \\
u_f^{(s)}&=f\;\;\textrm{in}\;\;\mathbb{R}^n\setminus B_1.
\end{split}
\end{cases}
\end{equation}
\end{thm} 

For a proof of Theorem~\ref{Thm} see e.g. Theorem 2.10 in \cite{1}.

We now generalize
Theorem \ref{Thm} by allowing external data that are not necessarily continuous:

\begin{prop}\label{MGSS}
Let $n\ge2$, $s\in (0,1)$, $R>1$ and~$f\in  L^\infty(B_{R}\setminus B_1)\cap
L_s^1(\mathbb{R}^n\setminus B_1)$. Then the function defined in~\eqref{Harms} is the unique solution (up to a zero measure subset of $\mathbb{R}^n\setminus B_1$) to the problem in~\eqref{oewrui35tu73485435u75465jhgdfjkgjd}.
\end{prop} 

\begin{proof} We argue by approximation, owing to Theorem \ref{Thm}.
The gist is indeed to take a sequence of continuous functions~$f_k$ approaching~$f$ as~$k\to+\infty$,
use Theorem \ref{Thm} and then pass to the limit. To implement this idea, one needs to take care
of some regularity issues.

The details of this technical argument go as follows. By~\eqref{BIS},
for each $x\in B_1$ and multi-index $\alpha$ we have that
\begin{equation*}
D_{x}^\alpha P(x,\cdot)f(\cdot)\in L^1(\mathbb{R}^n\setminus \overline{B}_1).
\end{equation*}
As a consequence, we obtain that~$u_f^{(s)}(x)$ in~\eqref{Harms} is well defined and smooth inside $B_1$. 

To complete the proof of Proposition~\ref{MGSS},
we need to show that $u_f^{(s)}$, as defined in~\eqref{Harms}, is the unique solution
of~\eqref{oewrui35tu73485435u75465jhgdfjkgjd}. 
To do so, we start by checking that~$u_f^{(s)}$ is $s$-harmonic in $B_1$. We consider a sequence $\lbrace f_k\rbrace_{k}\subset C(\mathbb{R}^n\setminus B_1)\cap L_s^1(\mathbb{R}^n\setminus B_1)$, such that 
\begin{equation}\label{EFK}
{\mbox{$f_k\rightarrow f$ in $L_s^1(\mathbb{R}^n\setminus B_1)$ as~$k\to+\infty$.}}\end{equation}
More specifically, we take $f_k:=(\chi_{B_k} \tilde{f})*\eta_{\frac{1}{k}}$
with $k\geq 2$, where $\tilde{f}$ is defined as
\begin{equation*}
\tilde{f}(x):=\begin{cases}
\begin{split}
f(x)\;\;\textrm{if}\;\;&x\in \mathbb{R}^n\setminus B_1,\\
0\;\;\textrm{if}\;\;&x\in B_1,
\end{split}
\end{cases}
\end{equation*}
and $\eta_{\frac{1}{k}}$ is a mollifier of radius $\frac{1}{k}$, while $\chi_{B_k}$ is the characteristic function of $B_k$. We also let~$u_{f_k}^{(s)}$ be the unique
pointwise continuous solution to the problem~\eqref{oewrui35tu73485435u75465jhgdfjkgjd},
according to Theorem~\ref{Thm}.

Then we have that for each multi-index $\alpha$
\begin{equation}\label{LocalUni}
\Vert D^\alpha u_{f_k}^{(s)}-D^\alpha u_{f}^{(s)}
\Vert_{L_{\textrm{loc}}^\infty(B_1)}\to 0 \quad {\mbox{ as~$k\to+\infty$}}. 
\end{equation}
Indeed for each multi-index $\alpha$ and $g\in L^\infty(B_R\setminus B_1)\cap
L_s^1(\mathbb{R}^n\setminus B_1)$ one finds that 
\begin{equation*}
D^\alpha u_{g}(x)=\int_{\mathbb{R}^n\setminus B_1}D_x^\alpha P(x,y)g(y)\,dy 
\end{equation*} 
for each $x\in B_1$, and therefore, choosing~$R_0\in(1,R)$, we see that, for every~$x\in B'$ with~$B'\Subset B_1$,
\begin{equation}\label{118}
\begin{split}&
\vert D^\alpha u_{f_k}^{(s)}(x)-D^\alpha u_{f}^{(s)}(x)\vert \\&\qquad\leq 
\int_{\mathbb{R}^n\setminus B_{R_0}}\vert D_x^\alpha P(x,y)\vert\vert f_k(y)-f(y)\vert \,dy+ 
\int_{B_{R_0}\setminus B_1}\vert D_x^\alpha P(x,y)\vert \vert f_k(y)-f(y)\vert \,dy\\
&\qquad\leq c\,\int_{\mathbb{R}^n\setminus B_{R_0}}\frac{\vert f_k(y)-f(y)\vert}{ \vert y\vert^{n+2s}}\,dy+ \int_{B_{R_0}\setminus B_1}\vert D_x^\alpha P(x,y)\vert \vert f_k(y)-f(y)\vert\, dy,
\end{split}
\end{equation}
where $c$ is a positive constant depending on $\alpha$, $R_0$, $n$, $s$ and~$B'$.
The first term in the third
line in \eqref{118} converges to zero as~$k\to+\infty$,
thanks to~\eqref{EFK}. 
We also
observe that, if~$y\in B_{R_0}\setminus B_1$, then~$ \vert f_k(y)-f(y)\vert \le 2
\Vert f\Vert_{L^\infty(B_R\setminus B_1)}$, and therefore, by the Dominated Convergence Theorem,
we have that also the second term in the third
line in \eqref{118} converges to zero as~$k\to+\infty$. These
considerations prove~\eqref{LocalUni}.

Furthermore note that if $\alpha=0$, taking~$R_0\in(1,R)$ and
using also~\eqref{TRE}, we have that, for all~$x\in B_{1}$,
\begin{equation}\label{Boundd}\begin{split}
\vert u_{f_k}^{(s)}(x)\vert \leq\;&\int_{\mathbb{R}^n\setminus B_{R_0}}
\vert P(x,y)\vert\vert f_k(y)\vert \,dy+ 
\int_{B_{R_0}\setminus B_1}\vert P(x,y)\vert \vert f_k(y)\vert \,dy\\
\leq \;& C\,\int_{\mathbb{R}^n\setminus B_{R_0}}\frac{\vert f_k(y)\vert}{ \vert y\vert^{n+2s}}\,dy
+ \int_{B_{R_0}\setminus B_1}\vert P(x,y)\vert \vert f_k(y)\vert\, dy\\
\le\;&C
\Vert f_k\Vert_{L_s^1(\mathbb{R}^n\setminus B_1)}+\Vert f \Vert_{L^\infty(B_R\setminus B_1)},
\end{split}\end{equation} 
where $C$ is a positive constant depending on $R_0$, $n$ and~$s$. Now, 
we observe that  the sequence $\Vert f_k\Vert_{L_s^1(\mathbb{R}^n
\setminus B_1)}$ is uniformly bounded,
thanks to~\eqref{EFK}. Accordingly, from~\eqref{Boundd}
we see that
\begin{equation}\label{peo249678ghdjgvbdh}
{\mbox{$u_{f_k}^{(s)}$ is uniformly bounded in~$ B_{1}$.}}\end{equation}

Now, if $x\in B_1$, taking~$\delta\in\left(0, 1-|x| \right)$,
we have that
\begin{equation}\label{wqw4358346y45y547777}
\begin{split}
(-\Delta)^s u_{f_k}^{(s)}(x)-(-\Delta)^s u_f^{(s)}(x)
=\;&\int_{\mathbb{R}^n}\frac{u_{f_k}^{(s)}(x)-u_{f_k}^{(s)}
(y)-u_f^{(s)}(x)+u_f^{(s)}(y)}{\vert x-y\vert ^{n+2s}}\,dy\\
=\;&A+B+C+D+E+F,
\end{split}
\end{equation}
where
\begin{eqnarray*}
&A:=\displaystyle\int_{\mathbb{R}^n\setminus B_1} 
\frac{u_{f_k}^{(s)}(x)-u_f^{(s)}(x)}{\vert x-y\vert^{n+2s}}\,dy,
\qquad\qquad&
B:=\int_{\mathbb{R}^n\setminus B_1}\frac{u_{f_k}^{(s)}
(y)-u_{f}^{(s)}(y)}{\vert x-y\vert^{n+2s}}\,dy ,\\
&C:=\displaystyle\int_{B_\delta(x)}\frac{u_{f_k}^{(s)}(x)-u_{f_k}^{(s)}(y)}{\vert x-y\vert^{n+2s}}\,dy,\qquad\qquad&
D:=\int_{B_\delta(x)}\frac{u_{f}^{(s)}(y)-u_{f}^{(s)}(x)}{\vert x-y\vert^{n+2s}}\,dy,\\
&E:=\displaystyle\int_{B_1\setminus B_\delta(x)}\frac{u_{f_k}^{(s)}
(x)-u_{f}^{(s)}(x)}{\vert x-y\vert^{n+2s}}\,dy\qquad
{\mbox{ and }}\qquad &F:=\int_{B_1\setminus B_\delta(x)}\frac{u_{f}^{(s)}
(y)-u_{f_k}^{(s)}(y)}{\vert x-y\vert^{n+2s}}\,dy.
\end{eqnarray*}
Notice that
\begin{equation*}
|A+E|
\le \int_{\mathbb{R}^n\setminus B_\delta(x)} 
\frac{ |u_{f_k}^{(s)} (x)- u_f^{(s)}(x)|}{\vert x-y\vert^{n+2s}}\,dy\le |u_{f_k}^{(s)} (x)- u_f^{(s)}(x)|
\int_{\R^n\setminus B_\delta}\frac{dz}{|z|^{n+2s}}\le \frac{C}{\delta^{2s}} |u_{f_k}^{(s)} (x)- u_f^{(s)}(x)|,
\end{equation*}
which converges to zero as~$k\to+\infty$,  thanks to~\eqref{LocalUni}.

Furthermore, we observe that if~$y\in\mathbb{R}^n\setminus B_1$ then
$$ |x-y|\ge |y|-|x|=\delta |y| +(1-\delta)|y|-|x|\ge \delta |y| +1-\delta-|x|\ge \delta |y|,
$$
and thus
\begin{equation*}
\vert B\vert \leq 
\int_{\mathbb{R}^n\setminus B_1}\frac{\vert f_k(y)-f(y)\vert}{ \vert x-y\vert^{n+2s}}\,dy\le
\frac{1}{\delta^{n+2s}}
\int_{\mathbb{R}^n\setminus B_1}\frac{\vert f_k(y)-f(y)\vert}{\vert y\vert^{n+2s}}\,dy 
\end{equation*}  
which, in light of~\eqref{EFK}, converges to zero as~$k\to+\infty$.

Moreover, from~\eqref{LocalUni} and the Dominated Convergence Theorem, we see that
the quantity~$C+D$ converges to zero as~$k\to+\infty$.

Finally, recalling~\eqref{peo249678ghdjgvbdh}
and making again use of the Dominated Convergence Theorem, we have that~$F$ converges to zero
as~$k\to+\infty$.

These considerations and~\eqref{wqw4358346y45y547777} give that
$(-\Delta)^s u_{f_k}^{(s)}(x)$ converges to~$(-\Delta)^s u_f^{(s)}(x)$ as~$k\to+\infty$
for every~$x\in B_1$.
Since, by Theorem~\ref{Thm}, we know that $(-\Delta)^su_{f_k}^{(s)}(x)=0$ for each $x\in B_1$,
we conclude that $(-\Delta)^s u_f^{(s)}(x)=0$. This proves that~$u_f^{(s)}$ solves~\eqref{oewrui35tu73485435u75465jhgdfjkgjd}.

It is only left to show the uniqueness statement. Suppose that there exists $u_1:\mathbb{R}^n\rightarrow\mathbb{R}$ satisfying
\begin{equation*}
\begin{cases}
\begin{split}
(-\Delta)^s u_1&=0\;\;\textrm{in}\;\; B_1, \\
u_1&=f\;\;\textrm{in}\;\;\mathbb{R}^n\setminus B_1.
\end{split}
\end{cases}
\end{equation*} 
Then both $v:=u_f^{(s)}-u_1$ and $-v=u_1- u_f^{(s)}$  are solutions to
\begin{equation*}
\begin{cases}
\begin{split}
(-\Delta)^s u&=0\;\;\textrm{in}\;\; B_1, \\
u&=0\;\;\textrm{in}\;\;\mathbb{R}^n\setminus B_1,
\end{split}
\end{cases}
\end{equation*}  
and therefore by the maximum principle for the fractional Laplacian (see e.g. Theorem~2.3.2. in~\cite{7}) we have that $v=0$ in $B_1$, leading to uniqueness.
\end{proof}

\section{Proof of the Fractional Malmheden 
and Schwarz Theorems}\label{TTTTT}

In this section we provide the proofs of the Fractional Malmheden 
and Schwarz results, as stated in Theorems~\ref{Ted} and~\ref{N=2}, and of the
convergence result in Proposition~\ref{sto1}.

We start with the main argument to prove Theorem~\ref{Ted}. For this, we
employ the following change of variable result (see Lemma 2.13.3 in~\cite{2} for the proof of it):

\begin{lem}\label{lemma}
Let $n\geq 2$, $x\in B_1$ and $Q_{\pm}^x$ and~$r_{\pm}^x$
be defined as in \eqref{12} and~\eqref{23}, respectively. Then for each $e\in \partial B_1$ it holds
that
\begin{equation*}
\vert \det D Q_{\pm}^x(e)\vert =\frac{(\pm r_{\pm}^x(e))^n}{1-\vert x\vert^2-r_{\pm}^x(e)x\cdot e},
\end{equation*}
and for each continuous $f:\partial B_1\rightarrow \mathbb{R} $ we have that
\begin{equation*}
\int_{\partial B_1} f(e)\,dH_e^{n-1} =\int_{\partial B_1} f(Q_{\pm}^x(e))\frac{(\pm r_{\pm}^x(e))^n}{1-\vert x\vert^2-r_{\pm}^x(e)x\cdot e} \,dH_e^{n-1}.
\end{equation*}
\end{lem}

With this notation, Theorem~\ref{Ted} will be a consequence of the following statement.

\begin{thm}\label{FMAHT}
Let $n\geq 2$,
$s\in (0,1)$, $R>1$
and $f\in L^\infty(B_R\setminus B_1)\cap L_s^1(\mathbb{R}^n\setminus B_1)$.
Let~$u_f^{(s)}$ be as in~\eqref{Harms}.

Then, for each $x\in B_1$,
\begin{equation}\label{someer}
u_f^{(s)}(x)=\int_{1}^\infty \left(
\int_{\partial B_1}{\mathcal{E}}(x,\rho)\,\mathcal{L}_{f,e,\rho}(x)\,dH_{e}^{n-1}\right)\,d\rho,
\end{equation} 
where the notation in~\eqref{NEMDKJA} and~\eqref{Edefin999} has been used.

Furthermore, if $f$ is positively homogeneous of degree $\gamma$ for some $\gamma\geq 0$, then we have that
\begin{equation}\label{someer2}
u_f^{(s)}(x)=\int_{1}^\infty\left( \int_{\partial B_1}\rho^\gamma{\mathcal{E}}(x,\rho)\,\mathcal{L}_{f,e,1}
\left(\frac{x}{\rho}\right)\,dH_{e}^{n-1}\right)\,d\rho.
\end{equation}
\end{thm}

\begin{proof}
We first suppose
that $f\in C(\mathbb{R}^n\setminus B_1)\cap L_s^1(\mathbb{R}^n\setminus B_1)$.
Let $x\in B_1$, then, using polar coordinates, from~\eqref{PoKe} and~\eqref{Harms}
we get that
\begin{equation}\label{eiwrheutherutghty576767687PPP}
\begin{split}
u_f^{(s)}(x)&=\int_{\mathbb{R}^n\setminus B_1}P(x,y)f(y)\,dy \\
&=c(n,s)(1-\vert x\vert^2)^s\int_{\mathbb{R}^n\setminus B_1}\frac{1}{(\vert y\vert^2-1)^s}\frac{f(y)}{\vert x-y\vert^n}\,dy\\
&=c(n,s)(1-\vert x\vert^2)^s \int_{1}^\infty \frac{1}{(\rho^2-1)^s}\left(
\int_{\partial B_\rho} \frac{f(\omega)}{\vert x-\omega\vert^n}\,dH_{\omega}^{n-1}\right)\,d\rho \\
&=c(n,s)(1-\vert x\vert^2)^s \int_{1}^\infty \frac{\rho^{n-1}}{(\rho^2-1)^s}\left(
\int_{\partial B_1} \frac{f(\rho e)}{\vert x-\rho e\vert^n}\,dH_{e}^{n-1}\right)\,d\rho\\
&=c(n,s)(1-\vert x\vert^2)^s\int_{1}^\infty \frac{1}{\rho (\rho^2-1)^s}\left(
\int_{\partial B_1} \frac{f(\rho e)}{\big| \frac{x}{\rho}-e\big|^n}\,dH_e^{n-1}\right)\,d\rho\\
&=c(n,s)(1-\vert x\vert^2)^s\int_{1}^\infty \frac{\rho}{(\rho^2-\vert x\vert^2) (\rho^2-1)^s}\left(\int_{\partial B_1}\frac{f(\rho e)}{\big| \frac{x}{\rho}-e\big|^n} \left(1-\frac{|x|^2}{\rho^2}\right)\,dH_e^{n-1}\right)\,d\rho\\
&=:{\mathcal{I}}.
\end{split}
\end{equation} Hence, defining 
$$g(e):=\frac{f(\rho e)}{\big| \frac{x}{\rho}-e\big|^n} \left(1-\frac{|x|^2}{\rho^2}\right)$$
and applying Lemma~\ref{lemma} we obtain that
\begin{equation}\label{RIUCOSLED}
\begin{split}&\frac{
{\mathcal{I}} }{c(n,s)(1-\vert x\vert^2)^s}\\&\quad=
\int_{1}^\infty \frac{\rho}{(\rho^2-\vert x\vert^2) (\rho^2-1)^s}\left(
\int_{\partial B_1}g(e)\,dH_{e}^{n-1}\right)\,d\rho\\
&\quad=\int_{1}^\infty \frac{\rho}{(\rho^2-\vert x\vert^2) (\rho^2-1)^s}\left(
\int_{\partial B_1}g(Q_{-}^{x/\rho}(e))\frac{(-r_{-}^{x/\rho}(e))^n}{1-\vert x/\rho\vert^2-(x/\rho\cdot e)r_{-}^{x/\rho}(e)}\,dH_e^{n-1}\right)\,d\rho\\
&\quad=\int_{1}^\infty \frac{\rho}{(\rho^2-\vert x\vert^2) (\rho^2-1)^s}\left(
\int_{\partial B_1}\frac{f(\rho Q_{-}^{x/\rho}(e))}{\big| \frac{x}{\rho}-Q_{-}^{x/\rho}(e)\big|^n} \frac{\left(1-\vert x/\rho\vert^2\right)(-r_{-}^{x/\rho}(e))^n}{1-\vert x/\rho\vert^2-(x/\rho\cdot e)r_{-}^{x/\rho}(e)}\,dH_e^{n-1}\right)
\,d\rho.\end{split}
\end{equation}
From equations \eqref{12} and \eqref{23} we deduce that
$$\left| \frac{x}{\rho}-Q_{-}^{x/\rho}(e)\right|=\vert r_{-}^{x/\rho}(e)\vert=-r_{-}^{x/\rho}(e),$$ and also
(see formula~(2.13.25) in~\cite{2})
\begin{equation*}
\frac{1-\vert x/\rho\vert^2}{1-\vert x/\rho\vert^2-(x/\rho\cdot e)r_{-}^{x/\rho}(e)}=\frac{2r_{+}^{x/\rho}(e)}{r_{+}^{x/\rho}(e)-r_{-}^{x/\rho}(e)}
\end{equation*}
which, together with~\eqref{RIUCOSLED}, gives that
\begin{equation}\label{djiewrteutyeiutyetyu}
{\mathcal{I}}=c(n,s)(1-\vert x\vert^2)^s\int_{1}^\infty \frac{\rho}{(\rho^2-\vert x\vert^2) (\rho^2-1)^s}
\left(\int_{\partial B_1}\frac{2r_{+}^{x/\rho}(e)f(\rho Q_{-}^{x/\rho}(e))}{r_{+}^{x/\rho}(e)-r_{-}^{x/\rho}(e)}
\,dH_e^{n-1}\right)\,d\rho.
\end{equation}
We now observe that 
\begin{equation}\label{KMDAAPKENIE}
\begin{split}
&\int_{\partial B_1}\frac{2r_{+}^{x/\rho}(e)f(\rho Q_{-}^{x/\rho}(e))}{r_{+}^{x/\rho}(e)-r_{-}^{x/\rho}(e)}\,dH_e^{n-1}\\
=&\int_{\partial B_1} \frac{r_{+}^{x/\rho}(e)f(\rho Q_{-}^{x/\rho}(e))}{r_{+}^{x/\rho}(e)-r_{-}^{x/\rho}(e)}\,dH_e^{n-1}-\int_{\partial B_1}\frac{r_{-}^{x/\rho}(e)f(\rho Q_{+}^{x/\rho}(e)))}{r_{+}^{x/\rho}(e)-r_{-}^{x/\rho}(e)}\,dH_e^{n-1}\\
=&\int_{\partial B_1}
\frac{r_{+}^{x/\rho}(e)f(\rho Q_{-}^{x/\rho}(e))-r_{-}^{x/\rho}(e)f(\rho Q_{+}^{x/\rho}(e))}{r_{+}^{x/\rho}(e)
-r_{-}^{x/\rho}(e)}\,dH_{e}^{n-1}.
\end{split}
\end{equation}
{F}rom \eqref{12}, \eqref{23}, \eqref{1.9BIS} and~\eqref{NEMDKJA} we also deduce that
\begin{eqnarray*}
&&\frac{r_{+}^{x/\rho}(e)f(\rho Q_{-}^{x/\rho}(e))-r_{-}^{x/\rho}(e)f(\rho Q_{+}^{x/\rho}(e))}{r_{+}^{x/\rho}(e)
-r_{-}^{x/\rho}(e)}\\&=&
\frac{r_{+}^{x/\rho}(e)f_\rho( Q_{-}^{x/\rho}(e))-r_{-}^{x/\rho}(e)f_\rho (Q_{+}^{x/\rho}(e))}{
|Q_{+}^{x/\rho}(e)- Q_{-}^{x/\rho}(e)|}\\&=&
\frac{\left(Q_{+}^{x/\rho}(e)-\frac{x}\rho\right)\cdot e\,
f_\rho( Q_{-}^{x/\rho}(e))+\left(\frac{x}\rho-Q_{-}^{x/\rho}(e)\right)
\cdot e\,f_\rho (Q_{+}^{x/\rho}(e))}{
|Q_{+}^{x/\rho}(e)- Q_{-}^{x/\rho}(e)|}\\&=&
 \mathcal{L}_{f_\rho}^{Q^{x/\rho}_-(e),Q^{x/\rho}_+(e)}\left(\frac{x}\rho\right)\\&=&
\mathcal{L}_{f,e,\rho}(x).
\end{eqnarray*}
Plugging this information into~\eqref{KMDAAPKENIE} we find that
$$\int_{\partial B_1}\frac{2r_{+}^{x/\rho}(e)f(\rho Q_{-}^{x/\rho}(e))}{r_{+}^{x/\rho}(e)-r_{-}^{x/\rho}(e)}\,dH_e^{n-1}
=\int_{\partial B_1}\mathcal{L}_{f,e,\rho}(x)
\,dH_{e}^{n-1},$$
which in turn, together with~\eqref{djiewrteutyeiutyetyu}, gives that
$${\mathcal{I}}=c(n,s)(1-\vert x\vert^2)^s\int_{1}^\infty \frac{\rho}{(\rho^2-\vert x\vert^2) (\rho^2-1)^s}
\left(\int_{\partial B_1}\mathcal{L}_{f,e,\rho}(x)
\,dH_e^{n-1}\right)\,d\rho.$$
Thus, recalling~\eqref{Edefin999}, this and~\eqref{eiwrheutherutghty576767687PPP} establish the desired result in~\eqref{someer}
under the additional assumption that $f$ is continuous in $\mathbb{R}^n\setminus B_1$. 

Now we remove the continuity assumption on $f$ by an approximation argument. Given~$f\in L^\infty(B_R\setminus B_1)\cap L_s^1(\mathbb{R}^n\setminus B_1)$, we consider a sequence of functions $\lbrace f_k\rbrace_k\subset C(\mathbb{R}^n\setminus B_1)\cap L_s^1(\mathbb{R}^n\setminus B_1)$ as in~\eqref{EFK}
and we let~$u_{f_k}^{(s)}$ be the unique
pointwise continuous solution to the problem~\eqref{oewrui35tu73485435u75465jhgdfjkgjd},
according to Theorem~\ref{Thm}.

By~\eqref{LocalUni},
\begin{equation}\label{EQ30}
\Vert u_{f_k}^{(s)}-u_f^{(s)}\Vert_{L_{\textrm{loc}}^\infty(B_1)}\to 0.
\end{equation}
Therefore we have that, for each $x\in B_1$,
\begin{equation*}
u_f^{(s)}(x)=\lim_{k\to \infty}\int_{1}^\infty\left(
\int_{\partial B_1} \mathcal{E}(x,\rho)\mathcal{L}_{f_k,e,\rho}(x)\,dH_{e}^{n-1}\right)\,d\rho.
\end{equation*}
{F}rom this, we claim that
there exists a subsequence $\lbrace f_{k_j}\rbrace_j$, such that 
\begin{equation}\label{Append}
\begin{split}
u_f^{(s)}(x)&=\lim_{j\to\infty}\int_{1}^\infty\left(
\int_{\partial B_1} \mathcal{E}(x,\rho)\mathcal{L}_{f_{k_j},e,\rho}(x)\,dH_{e}^{n-1}\right)\,d\rho\\
&=\int_{1}^\infty\left(\int_{\partial B_1} \mathcal{E}(x,\rho)\mathcal{L}_{f,e,\rho}(x)\,dH_{e}^{n-1}\right)\,d\rho.
\end{split}
\end{equation}
For the convenience of the reader,
the technical proof of \eqref{Append} can be found in Appendix~\ref{po2}.

The claim in~\eqref{Append} establishes the desired result in~\eqref{someer}
also for non continuous functions.

To prove the claim in~\eqref{someer2}, we use~\eqref{FORMULR} and~\eqref{NEMDKJA} to see that
\begin{eqnarray*}
\mathcal{L}_{f,e,\rho}(x)&=&
\mathcal{L}_{f_\rho}^{Q^{x/\rho}_-(e),Q^{x/\rho}_+(e)}\left(\frac{x}{\rho}\right) \\
&=&\frac{\left(\frac{x}\rho-Q^{x/\rho}_-(e)
\right)\cdot e}{|Q^{x/\rho}_+(e)-Q^{x/\rho}_-(e)|}\,f_\rho\big(Q^{x/\rho}_+(e)\big)
+\frac{\left(Q^{x/\rho}_+(e)-\frac{x}\rho\right)\cdot e}{|Q^{x/\rho}_+(e)-Q^{x/\rho}_-(e)|}\,f_\rho\big(
Q^{x/\rho}_-(e)\big)\\
&=&
\frac{\left(\frac{x}\rho-Q^{x/\rho}_-(e)
\right)\cdot e}{|Q^{x/\rho}_+(e)-Q^{x/\rho}_-(e)|}\,f \big(\rho Q^{x/\rho}_+(e)\big)
+\frac{\left(Q^{x/\rho}_+(e)-\frac{x}\rho\right)\cdot e}{|Q^{x/\rho}_+(e)-Q^{x/\rho}_-(e)|}\,f \big(\rho
Q^{x/\rho}_-(e)\big)\\&=&
\rho^\gamma\left[
\frac{\left(\frac{x}\rho-Q^{x/\rho}_-(e)
\right)\cdot e}{|Q^{x/\rho}_+(e)-Q^{x/\rho}_-(e)|}\,f \big(  Q^{x/\rho}_+(e)\big)
+\frac{\left(Q^{x/\rho}_+(e)-\frac{x}\rho\right)\cdot e}{|Q^{x/\rho}_+(e)-Q^{x/\rho}_-(e)|}\,f \big( 
Q^{x/\rho}_-(e)\big)\right]
\\&=& \rho^\gamma
{\mathcal{L}}_f^{ Q^{x/\rho}_-(e),  Q^{x/\rho}_+(e)}\left(\frac{x}\rho\right)\\&=&
 \rho^\gamma
{\mathcal{L}}_{f,e,1}\left(\frac{x}\rho\right)
.
\end{eqnarray*}
The claim in~\eqref{someer2} then follows from this and~\eqref{someer}.
This concludes the proof of Theorem~\ref{FMAHT}. 
\end{proof}

\begin{proof}[Proof of Theorem~\ref{Ted}]
{F}rom Proposition~\ref{MGSS}, we know that, under the hypotheses of Theorem~\ref{Ted},
the function defined in~\eqref{Harms} is the unique solution (up to a zero measure subset of $\mathbb{R}^n\setminus B_1$) to the problem~\eqref{FRDIPR}.
Then, the desired result in Theorem~\ref{Ted} follows from this and
Theorem~\ref{FMAHT}.
\end{proof}

We now give the proof of the Schwarz result in Theorem~\ref{N=2}.

\begin{proof}[Proof of Theorem~\ref{N=2}]
We first suppose that $f\in C(\mathbb{R}^2\setminus B_1)\cap L_s^1(\mathbb{R}^2\setminus B_1)$. Under these assumptions we can apply Theorem~\ref{sharashar} in dimension~$2$ and get 
\begin{equation*}
u_f^{(s)}(x)=2\pi \int_{1}^\infty\mathcal{E}(x,\rho)\tilde{u}_{f_\rho}\left(\frac{x}{\rho}\right)\,d\rho,
\end{equation*} 
where $\tilde{u}_{f_\rho}$ has been defined in the statement of Theorem \ref{sharashar}. Therefore, when $n=2$ we can apply Theorem \ref{SCHW} to $\tilde{u}_{f_\rho}$ and get 
\begin{equation*}
\tilde{u}_{f_\rho}\left(\frac{x}{\rho}\right)=\frac{1}{2\pi}\int_{\partial B_1}f_\rho( Q^{x/\rho}(e))\,dH_e^{n-1},
\end{equation*}
which leads to~\eqref{SBCLM} in the case in which~$f\in C(\mathbb{R}^2\setminus B_1)$. 

Suppose now that $f\in L^\infty(B_{R}\setminus B_1)\cap L_s^1(\mathbb{R}^2\setminus B_1)$, 
and consider a sequence of functions $\lbrace f_k\rbrace_k\subset C(\mathbb{R}^2\setminus B_1)\cap L_s^1(\mathbb{R}^2\setminus B_1)$ as in~\eqref{EFK}.
We let~$u_{f_k}^{(s)}$ be the unique
pointwise continuous solution to the problem~\eqref{oewrui35tu73485435u75465jhgdfjkgjd},
according to Theorem~\ref{Thm}. {F}rom the previous step, we have that, for each $x\in B_1$,  
$$u_{f_k}^{(s)}(x) =\int_{1}^\infty\mathcal{E}(x,\rho)\left(\int_{\partial B_1}f_{k}(\rho Q^{x/\rho}(e))\,dH_e^{n-1}
\right)\,d\rho.$$
By this and~\eqref{EQ30}, we have that, for each $x\in B_1$,  
\begin{equation*}
u_f^{(s)}(x)=\lim_{k\to \infty}\int_{1}^\infty\left(
\int_{\partial B_1}\mathcal{E}(x,\rho)f_k(\rho Q^{x/\rho}(e))\,
dH_e^{n-1}\right)\,d\rho.
\end{equation*}
{F}rom this, one sees that
there exists a subsequence $\lbrace f_{k_j}\rbrace_{j}$ such that  
\begin{equation}\label{APPEND2}
\begin{split}
u_f^{(s)}(x)&=\lim_{j\to \infty}
\int_{1}^\infty\left(\int_{\partial B_1}
\mathcal{E}(x,\rho)f_{k_j}(\rho Q^{x/\rho}(e))\,dH_e^{n-1}\right)
\,d\rho\\
&=\int_{1}^\infty\left(
\int_{\partial B_1}\mathcal{E}(x,\rho)f(\rho Q^{x/\rho}(e))\,dH_e^{n-1}
\right)\,d\rho.
\end{split}
\end{equation}
For the facility of the reader, a detailed proof of~\eqref{APPEND2} is given in
Appendix~\ref{po3}.

We observe that the proof of Theorem~\ref{N=2} is completed, thanks to~\eqref{APPEND2}.
\end{proof}

We now deal with the convergence result in Proposition \ref{sto1}.

\begin{proof}[Proof of Proposition \ref{sto1}]
Let $f\in  C(B_R\setminus B_1)\cap L_s^1(\mathbb{R}^n\setminus B_1)$, for each $s\in (s_0,1]$.
Furthermore let~$u_f^{(s)}$ and~${u}_f$ as in the statement of Proposition~\ref{sto1}.
Then, Theorem~\ref{MahT} implies that the following identity holds for each $\rho\in (1,R)$
\begin{equation}\label{emir}
{u}_{f_\rho}\left(\frac{x}{\rho}\right)=
\fint_{\partial B_1}\mathcal{L}_{f_\rho}^{Q^{x/\rho}_-(e),Q^{x/\rho}_+(e)}\left(\frac{x}\rho\right)\,dH_{e}^{n-1} =
\fint_{\partial B_1}\mathcal{L}_{f,e,\rho}(x)\,dH_{e}^{n-1},
\end{equation} 
thanks to~\eqref{NEMDKJA},
where $ {u}_{f_\rho}$ is the unique solution to the classical Dirichlet problem \eqref{kurt}.

In particular, we have that
\begin{equation}\label{emirBIS}
{u}_{f }\left(\frac{x}{\rho}\right)=
\fint_{\partial B_1}\mathcal{L}_{f }^{Q^{x/\rho}_-(e),Q^{x/\rho}_+(e)}\left(\frac{x}\rho\right)\,dH_{e}^{n-1} =
\fint_{\partial B_1}\mathcal{L}_{f,e,1}\left(\frac{x}\rho\right)\,dH_{e}^{n-1}.
\end{equation}

Now, using~\eqref{141021}, \eqref{NormE}
and~\eqref{emir} we obtain for each $x\in B_1$ and $R_0\in(1,R)$ the following identity
\begin{equation}\label{Wax}
\begin{split}&
u_f^{(s)}(x)- {u}_f(x)\\
=\;&\int_{1}^\infty \vert \partial
B_1\vert \mathcal{E}(x,\rho)\left( \fint_{\partial B_1}\mathcal{L}_{f,e,\rho}(x)\,dH_e^{n-1}- {u}_f(x)\right)\,d\rho\\
=\;&\int_{1}^{R_0} \vert \partial B_1\vert \mathcal{E}(x,\rho)\left({u}_{f_\rho}\left(\frac{x}{\rho}\right)
-{u}_f(x)\right)
\,d\rho
\\&\qquad+\int_{R_0}^\infty\vert \partial
B_1\vert \mathcal{E}(x,\rho)\left( \fint_{\partial B_1}\mathcal{L}_{f,e,\rho}(x)\,dH_e^{n-1}- {u}_f(x)\right)\,d\rho.
\end{split}
\end{equation}

By the continuity of $ {u}_f$ and $f$, we have that for each $\delta>0$ there exists some $R_0\in (1,R)$ such that for each $\rho\in(1,R_0)$
\begin{equation}\label{spowr49ty4nvrughrguer7t6657}\begin{split}
&\left| {u}_f\left(\frac{x}{\rho}\right)-u_f(x)\right| \leq \delta
\quad {\mbox{ for all }} x\in B_1 \\
\qquad {\mbox{ and }}&\qquad
\Vert f_\rho-f\Vert_{L^\infty(\partial B_1)}\leq \delta.
\end{split}\end{equation}
Also, we point out that
\begin{eqnarray*}
\mathcal{L}_{f_\rho,e,1}\left(\frac{x}\rho\right)&=&
\mathcal{L}_{f_\rho }^{Q^{x/\rho}_-(e),Q^{x/\rho}_+(e)}\left(\frac{x}\rho\right)=
{\mathcal{L}}_{f,e,\rho}\left(x\right),
\end{eqnarray*}
thanks to~\eqref{NEMDKJA}.

Therefore, from this, \eqref{emir}, \eqref{emirBIS}, and~\eqref{spowr49ty4nvrughrguer7t6657}, 
for all~$x\in B_1$, we deduce that, if~$\rho\in(1,R_0)$,
\begin{equation}\label{WMPRE99}
\begin{split}
\left|  {u}_{f_\rho}\left(\frac{x}{\rho}\right)-{u}_{f}(x)\right|&\leq \left|
{u}_{f_\rho}\left(\frac{x}{\rho}\right)- {u}_f\left(\frac{x}{\rho}\right)\right|
 +\left|  {u}_f\left(\frac{x}{\rho}\right)-{u}_f(x)\right| \\
&=\left| \fint_{\partial B_1} \mathcal{L}_{f,e,\rho}(x)-\mathcal{L}_{f,e,1}\left(\frac{x}{\rho}\right)\,dH_{e}^{n-1} \right|
+\left|  {u}_f\left(\frac{x}{\rho}\right)- {u}_f(x)\right| \\
&=\left| \fint_{\partial B_1} \mathcal{L}_{f_\rho-f,e,1}\left(\frac{x}\rho\right)\,dH_e^{n-1}\right|+
\left|  {u}_f\left(\frac{x}{\rho}\right)- {u}_f(x)\right| \\
&\leq \fint_{\partial B_1} \left| \mathcal{L}_{f_\rho-f,e,1}\left(\frac{x}{\rho}\right)\right|\,dH_e^{n-1} +\delta.
\end{split}
\end{equation}
Now, if~$x\in B_1$ and~$\rho\in(1,R_0)$, we see that
\begin{equation}\label{sowru38tu348t43ty48y458tyhghefweigheo}\begin{split}
&|Q^{x/\rho}_+(e)-Q^{x/\rho}_-(e)|= 2\sqrt{\left(\frac{x}\rho\cdot e\right)^2-\frac{|x|^2}{\rho^2}+1}
\ge 2\sqrt{1-\frac{|x|^2}{\rho^2} }\\&\qquad\qquad
=\frac{2}{\rho} \sqrt{\rho^2- |x|^2  }\ge\frac{2}{R_0} \sqrt{1- |x|^2  }
\end{split}\end{equation}
and thus, according to~\eqref{FORMULR},
\begin{equation}\label{swireyeghreu50978659650}\begin{split}&
\left|\mathcal{L}_{f_\rho-f,e,1}\left(\frac{x}\rho\right)\right|
\\=\;&
\left| \frac{\left(\frac{x}\rho-Q^{x/\rho}_-(e)\right)\cdot e}{|Q^{x/\rho}_+(e)-Q^{x/\rho}_-(e)|}
\,(f_\rho-f)\big(Q^{x/\rho}_+(e)\big)+\frac{\left(Q^{x/\rho}_+(e)-\frac{x}\rho\right)
\cdot e}{|Q^{x/\rho}_+(e)-Q^{x/\rho}_-(e)|}\,(f_\rho-f)\big(Q^{x/\rho}_-(e)\big) \right|\\
\leq\;& 4 \,
\frac{R_0}{2 \sqrt{1- |x|^2  }} \Vert  f_\rho-f \Vert_{L^\infty(\partial B_1)}
\\ \leq\;& \frac{2R_0\,\delta }{\sqrt{1-\vert x\vert^2}},
\end{split}\end{equation}
and therefore, plugging this information into~\eqref{WMPRE99}, we obtain that, if~$x\in B_1$
and~$\rho\in(1, R_0)$,
\begin{equation}\label{WM}
\left|  {u}_{f_\rho}\left(\frac{x}{\rho}\right)-{u}_{f}(x)\right| \leq \left( \frac{2R_0  }{\sqrt{1-\vert x\vert^2}}
+1\right)\delta.
\end{equation}

Furthermore, employing the
change of variable~$e:=\omega/|\omega|$ and recalling~\eqref{Edefin999},
\begin{equation}\label{Hamilton}
\begin{split}
&  \int_{R_0}^\infty\mathcal{E}(x,\rho)\left(\int_{\partial B_1}\mathcal{L}_{f,e,\rho}(x)\,dH_e^{n-1}-
\vert \partial B_1\vert {u}_f(x)\right)\,d\rho\\ 
=\;& \int_{R_0}^\infty \frac{\mathcal{E}(x,\rho)}{\rho^{n-1}}\left(\int_{\partial B_\rho}\mathcal{L}_{f,\omega/\vert \omega\vert,\vert \omega\vert}(x)\,dH_\omega^{n-1}-\vert \partial B_1\vert {u}_f(x)\right)\,d\rho\\ 
=\;& c(n,s)(1-\vert x\vert^2)^s \int_{\mathbb{R}^n\setminus B_{R_0}} \frac{\vert y\vert^2}{(\vert y\vert^2-1)^s(\vert y\vert^2-\vert x\vert^2)\vert y\vert^n} \left(\mathcal{L}_{f,y/\vert y\vert,\vert y\vert}(x)-\vert\partial B_1\vert
{u}_f(x)\right)dy.
\end{split}
\end{equation}

We also deduce from~\eqref{FORMULR} the following pointwise estimate
\begin{equation}\label{dadaism}\begin{split}&
\left|\mathcal{L}_{f,y/\vert y\vert,\vert y\vert}(x)\right|\\=\;&\left|
\frac{\left(\frac{x}{|y|}-Q_{-}^{x/\vert y\vert}\left(\frac{y}{\vert y\vert}\right)\right)\cdot
\frac{y}{|y|}}{|Q_{+}^{x/\vert y\vert}\left(\frac{y}{\vert y\vert}\right)-Q_{-}^{x/\vert y\vert}\left(\frac{y}{\vert y\vert}\right)|}\,f_{|y|}\left(Q_{-}^{x/\vert y\vert}\left(\frac{y}{\vert y\vert}\right)\right)
+\frac{\left(Q_{+}^{x/\vert y\vert}\left(\frac{y}{\vert y\vert}\right)
-\frac{x}{|y|}\right)\cdot \frac{y}{|y|}}{|Q_{+}^{x/\vert y\vert}\left(\frac{y}{\vert y\vert}\right)-Q_{-}^{x/\vert y\vert}\left(\frac{y}{\vert y\vert}\right)|}\,f_{|y|}\left(Q_{-}^{x/\vert y\vert}\left(\frac{y}{\vert y\vert}\right)\right)\right|
\\&\qquad
 \leq \frac{1}{\sqrt{1-\vert x\vert^2}}\left(\left| f_{\vert y\vert}\left(Q_{-}^{x/\vert y\vert}\left(
 \frac{y}{\vert y\vert}\right)\right)\right|+ \left| f_{\vert y\vert}\left(Q_{+}^{x/\vert y\vert}\left(
 \frac{y}{\vert y\vert}\right)\right) \right|\right).\end{split}
\end{equation}
We claim that
\begin{equation}\label{jdweirfyheue}
{\mbox{the right hand side in \eqref{dadaism} is $L_s^1(\mathbb{R}^n\setminus B_1)$ for each $s\in (s_0,1]$.}}\end{equation} 
Indeed, if we define the following function 
\begin{equation}\label{mappaF}
\begin{split}
F_{\pm}:\mathbb{R}^n\setminus &B_1\rightarrow \mathbb{R}^n\setminus B_1\\
&y\mapsto \vert y\vert Q_{\pm}^{x/\vert y\vert}(y/\vert y\vert), 
\end{split}
\end{equation}
we see that it is $C^1$ and invertible. Note that, recalling also the limits in \eqref{mandorla}, one finds that 
\begin{equation}\label{Royhargrove}
\Vert \det(DF_{\pm}^{-1}(z)\Vert_{L^\infty(\mathbb{R}^n\setminus B_1)}\leq C,
\end{equation}
for some positive constant $C$, depending on $x$. 
Therefore by applying Theorem~2 in Section~3.3.3 of~\cite{3}, we obtain
that 
\begin{equation}\label{Ray}\begin{split}
&\int_{\mathbb{R}^n\setminus B_1} \frac{\left| f_{\vert y\vert}
\left(Q_{\pm}^{x/\vert y\vert}(y/\vert y\vert)\right)\right|}{ 
\left|  y \right|^{n+2s}}
\,dy=
\int_{\mathbb{R}^n\setminus B_1} \frac{\left| f_{\vert y\vert}
\left(Q_{\pm}^{x/\vert y\vert}(y/\vert y\vert)\right)\right|}{ 
\left| \vert y\vert Q_{\pm}^{x/\vert y\vert}(y/\vert y\vert)\right|^{n+2s}}
\,dy\\&\qquad\qquad
=\int_{\mathbb{R}^n\setminus B_1}\frac{\vert f(z)\vert}{
\vert z\vert^{n+2s}}\vert\det(DF_{\pm}^{-1}(z))\vert\,dz
\le C \int_{\mathbb{R}^n\setminus B_1}\frac{\vert f(z)\vert}{
\vert z\vert^{n+2s}} \,dz
.\end{split}
\end{equation} 
This and the fact that~$f\in L_s^1(\mathbb{R}^n\setminus B_1)$ for each $s\in (s_0,1]$
entail that~\eqref{jdweirfyheue} holds true.

As a consequence of~\eqref{dadaism}
and~\eqref{jdweirfyheue} we have that the integrals in~\eqref{Hamilton} are finite and bounded
in~$s$.

Using this information and~\eqref{WM}, we deduce from~\eqref{Wax} that
for each $\delta>0$ there exists some~$R_0\in (1,R)$ such that for each $\rho\in (1,R_0)$ we have
\begin{equation*}
\begin{split}&
\vert u_f^{(s)}(x)-{u}_f(x)\vert \\ \le\;&
\int_{1}^{R_0} \vert \partial B_1\vert \mathcal{E}(x,\rho)\left|{u}_{f_\rho}\left(\frac{x}{\rho}\right)
-{u}_f(x)\right|
\,d\rho
\\&\qquad\qquad+\left|\int_{R_0}^\infty\vert \partial
B_1\vert \mathcal{E}(x,\rho)\left( \fint_{\partial B_1}\mathcal{L}_{f,e,\rho}(x)\,dH_e^{n-1}- {u}_f(x)\right)\,d\rho
\right|\\ \le\;&
\int_{1}^{R_0} \vert \partial B_1\vert \mathcal{E}(x,\rho)\left( \frac{2R_0  }{\sqrt{1-\vert x\vert^2}}
+1\right)\delta
\,d\rho
\\&\qquad\qquad+\left|\int_{R_0}^\infty\vert \partial
B_1\vert \mathcal{E}(x,\rho)\left( \fint_{\partial B_1}\mathcal{L}_{f,e,\rho}(x)\,dH_e^{n-1}- {u}_f(x)\right)\,d\rho
\right|\\
\leq\;&  C(x,R)\delta+
c(n,s)(1-\vert x\vert^2)^s \int_{\mathbb{R}^n\setminus B_{R_0}} \frac{\vert y\vert^2}{(\vert y\vert^2-1)^s(\vert y\vert^2-\vert x\vert^2)\vert y\vert^n} \left(\mathcal{L}_{f,y/\vert y\vert,\vert y\vert}(x)-\vert\partial B_1\vert
{u}_f(x)\right)dy\\
\le \;&  C(x,R)\delta +c(n,s)(1-\vert x\vert^2)^s \int_{\mathbb{R}^n\setminus B_{R_0}} \frac{\vert y\vert^2}{(\vert y\vert^2-1)^s(\vert y\vert^2-\vert x\vert^2)\vert y\vert^n} \\&\qquad \times\left( \frac{1}{\sqrt{1-\vert x\vert^2}}
\left(\left| f_{\vert y\vert}\left(Q_{-}^{x/\vert y\vert}\left(
 \frac{y}{\vert y\vert}\right)\right)\right|+ \left| f_{\vert y\vert}\left(Q_{+}^{x/\vert y\vert}\left(
 \frac{y}{\vert y\vert}\right)\right) \right|\right)
-\vert\partial B_1\vert
{u}_f(x)\right)dy
\end{split}
\end{equation*} 
where $ {C}(x,R_0 )$ depends only on $x$ and~$R_0$.

By taking the limit as~$s\nearrow 1$, we see that 
\begin{eqnarray*}&&
\lim_{s\nearrow 1}c(n,s)(1-\vert x\vert^2)^s \int_{\mathbb{R}^n\setminus B_{R_0}} \frac{\vert y\vert^2}{(\vert y\vert^2-1)^s(\vert y\vert^2-\vert x\vert^2)\vert y\vert^n} \\&&\qquad \times\left( \frac{1}{\sqrt{1-\vert x\vert^2}}
\left(\left| f_{\vert y\vert}\left(Q_{-}^{x/\vert y\vert}\left(
 \frac{y}{\vert y\vert}\right)\right)\right|+ \left| f_{\vert y\vert}\left(Q_{+}^{x/\vert y\vert}\left(
 \frac{y}{\vert y\vert}\right)\right) \right|\right)
-\vert\partial B_1\vert
{u}_f(x)\right)dy=0
\end{eqnarray*}
since $c(n,s)\to 0$ for $s\nearrow 1$ by~\eqref{CNS}.
As a consequence
$$ \lim_{s\nearrow1}\vert u_f^{(s)}(x)-{u}_f(x)\vert \le  C(x,R)\delta.$$
This and the arbitrariness of~$\delta$ give the desired claim in Proposition~\ref{sto1}.
\end{proof}

\section{Harnack inequality}\label{Har}
In this section we provide a new proof of the Harnack inequality for $s$-harmonic functions as stated in Theorem~\ref{JHRGNA:THM}.
Our strategy is to use the Fractional Malmheden Theorem to show that this result can be directly inferred from the classical Harnack inequality for harmonic functions. 

\begin{proof}[Proof of Theorem~\ref{JHRGNA:THM}]
For convenience we call $u|_{\mathbb{R}^n\setminus B_1}=f$. Let us first assume that $f\in C(\mathbb{R}^n\setminus B_1)\cap L_s^1(\mathbb{R}^n\setminus B_1)$. Under this assumption, we can apply Theorem \ref{sharashar} and obtain that
\begin{equation*}
u(x)=\vert \partial B_1\vert\int_{1}^\infty \mathcal{E}(x,\rho) {u}_{f_\rho}\left(\frac{x}{\rho}\right)\,d\rho,
\end{equation*}
for each $x\in B_1$, where $ {u}_{f_\rho}$ has been defined in the statement of Theorem~\ref{sharashar}. Therefore, we have that
\begin{equation}\label{ide}
u(0)=c(n,s)|\partial B_1|\int_{1}^\infty\frac{ {u}_{f_\rho}(0)}{\rho(\rho^2-1)^s}\,d\rho.
\end{equation}

Now we fix $r\in(0,1)$ and we consider $x\in B_r$. Applying the Harnack inequality for classical harmonic functions to~$ {u}_{f_\rho}$, we have that
\begin{equation*}
 {u}_{f_\rho}(0)\leq \frac{(1+\vert x\vert/\rho)^{n-1}}{1-\vert x\vert/\rho} {u}_{f\rho}\left(\frac{x}{\rho}\right).
\end{equation*}
{F}rom this, \eqref{Edefin999} and~\eqref{ide} we obtain that
\begin{equation}\begin{split}\label{38r734857485473879784368hgfudhgdkghfdkj}
u(0)\leq\;&
c(n,s)|\partial B_1|\int_{1}^\infty\frac{ 1}{\rho(\rho^2-1)^s}
\frac{(1+\vert x\vert/\rho)^{n-1}}{1-\vert x\vert/\rho}\, {u}_{f\rho}\left(\frac{x}{\rho}\right)
\,d\rho\\=\;&
|\partial B_1|\int_{1}^\infty \mathcal{E}(x,\rho)
\frac{ (\rho^2-\vert x\vert^2)}{ \rho^2\,(1-\vert x\vert^2)^s}
\frac{(\rho+\vert x\vert )^{n-1}}{\rho^{n-2}(\rho-\vert x\vert )} \,{u}_{f\rho}\left(\frac{x}{\rho}\right)
\,d\rho\\=\;&
|\partial B_1|\int_{1}^\infty \mathcal{E}(x,\rho)
\frac{(\rho+\vert x\vert )^{n}}{\rho^{n}  (1-\vert x\vert^2)^s}\, {u}_{f\rho}\left(\frac{x}{\rho}\right)
\,d\rho
\\=\;& 
\vert \partial B_1\vert \int_{1}^\infty  \mathcal{E}(x,\rho) \,g(\rho,t)\,{u}_{f_\rho}\left(\frac{x}{\rho}\right)\,d\rho,
\end{split}\end{equation}
where for convenience we have called $t:=\vert x\vert$ in the last line and defined 
\begin{equation*}
g(\rho,t):=\frac{(\rho+t)^{n}}{\rho^n(1-t^2)^s},
\end{equation*}
with $(\rho,t)\in [1,\infty)\times [0,r]$. 

Since $g(\rho,t)$ is decreasing in $\rho$ and increasing in $t$, we have that 
\begin{equation*}
\frac{(1+r)^n}{(1-r^2)^s}=\sup_{(\rho,t)\in [1,\infty)\times[0,r]}g(\rho,t). 
\end{equation*}
Therefore, from this, \eqref{WEIGHT}
and~\eqref{38r734857485473879784368hgfudhgdkghfdkj} we obtain that
\begin{eqnarray*}
u(0)\le \vert \partial B_1\vert \frac{(1+r)^n}{(1-r^2)^s}
\int_{1}^\infty  \mathcal{E}(x,\rho)  \,{u}_{f_\rho}\left(\frac{x}{\rho}\right)\,d\rho=
\frac{(1+r)^n}{(1-r^2)^s} u(x),
\end{eqnarray*}
which establishes the first inequality in \eqref{FracHarn}. 

To prove the second inequality in \eqref{FracHarn}, 
we make use of the Harnack inequality for harmonic functions, thus obtaining that
\begin{equation*}
 {u}_{f_\rho}\left(\frac{x}{\rho}\right)\leq \frac{1+\vert x\vert/\rho}{(1-\vert x\vert/\rho)^{n-1}} {u}_{f_\rho}(0).
\end{equation*}
Using this and~\eqref{Edefin999}  into~\eqref{ide}, we find that
\begin{eqnarray*}
u(0)&=& c(n,s)|\partial B_1|\int_{1}^\infty\frac{ {u}_{f_\rho}(0)}{\rho(\rho^2-1)^s}\,d\rho\\
&\ge& c(n,s)|\partial B_1|\int_{1}^\infty\frac{1}{\rho(\rho^2-1)^s}
\frac{(1-\vert x\vert/\rho)^{n-1}}{1+\vert x\vert/\rho} \,
{u}_{f_\rho}\left(\frac{x}{\rho}\right)\,d\rho\\&=&
c(n,s)|\partial B_1|\int_{1}^\infty 
\frac{(\rho-\vert x\vert )^{n-1}}{\rho^{n-1} (\rho^2-1)^s(\rho+\vert x\vert )} \,
{u}_{f_\rho}\left(\frac{x}{\rho}\right)\,d\rho\\&=&
|\partial B_1|\int_{1}^\infty  \mathcal{E}(x,\rho) 
\frac{(\rho-\vert x\vert )^{n}}{\rho^{n} (1-|x|^2)^s} \,
{u}_{f_\rho}\left(\frac{x}{\rho}\right)\,d\rho.
\end{eqnarray*}
Using again
the notation~$t:=\vert x\vert$, we define the following function
\begin{equation}\label{42077}
g_1(\rho,t):=\frac{(\rho-t)^n}{\rho^n(1-t^2)^s},
\end{equation}
with $(\rho,t)\in [1,\infty)\times [0,r]$, and we see that
\begin{equation}\label{sw39r534784yu8hyguhvdsdkworu4iyh}
u(0)\ge |\partial B_1|\int_{1}^\infty  \mathcal{E}(x,\rho) 
g_1(\rho,t)\,
{u}_{f_\rho}\left(\frac{x}{\rho}\right)\,d\rho.\end{equation}

Since~$g_1$ is increasing in $\rho$, we have that, for all~$(\rho,t)\in [1,\infty)\times [0,r]$,
$$ g_1(\rho,t)\ge g_1(1,t)=\frac{(1-t)^n}{(1-t^2)^s}=\frac{(1-t)^{n-s}}{(1+t)^s}=: g_2(t).$$
Notice also that~$g_2$ is decreasing, and therefore, for all~$(\rho,t)\in [1,\infty)\times [0,r]$,
$$ g_1(\rho,t)\ge g_2(r)=\frac{(1-r)^{n-s}}{(1+r)^s}=\frac{(1-r)^n}{(1-r^2)^s}
.$$
Plugging this information into~\eqref{sw39r534784yu8hyguhvdsdkworu4iyh}
and recalling~\eqref{WEIGHT}, we get
$$u(0)\ge |\partial B_1|\frac{(1-r)^n}{(1-r^2)^s}\int_{1}^\infty  \mathcal{E}(x,\rho)\,
{u}_{f_\rho}\left(\frac{x}{\rho}\right)\,d\rho\ge \frac{(1-r)^n}{(1-r^2)^s}u(x),$$
which completes the proof
of~\eqref{FracHarn} under the additional continuity assumption
on~$f$. 

To deal with the general case, we perform an approximation argument,
whose details go as follows.
If $f\in L^\infty(B_R\setminus B_1)\cap L_s^1(\mathbb{R}^n\setminus B_1)$, we take a sequence $\lbrace f_k\rbrace_k\subset C(\mathbb{R}^n\setminus B_1)\cap L_s^1(\mathbb{R}^n\setminus B_1)$ as in~\eqref{EFK}.
Then for $u_{f_k}^{(s)}$ the two-sided inequality in \eqref{FracHarn} holds true, thanks to the previous step. 
Also, by~\eqref{EQ30}, we have the local uniform convergence
\begin{equation*}
\Vert u_{f_k}^{(s)}-u\Vert_{L_{\textrm{loc}}^\infty(B_1)}\to 0 \quad{\mbox{ as }}k\to+\infty,
\end{equation*}
and therefore we deduce the two sided inequality \eqref{FracHarn} also in this case.

It is only left to prove that the constants provided in equation \eqref{FracHarn} are optimal. To show this let us fix some direction $e\in \partial B_1$, a constant $\epsilon>0$ and the function
\begin{equation*}
f_\epsilon(y):=\begin{cases}
\begin{split}
0\;\;\;\;\;\;\;\;\;\;\;\;\;\;&\textrm{if}\;\;y\in \mathbb{R}^n\setminus B_{\epsilon}((1+\epsilon)e),\\
(\vert y\vert^2-1)^s\;\;&\textrm{if}\;\;y\in B_\epsilon((1+\epsilon)e).
\end{split}
\end{cases}
\end{equation*}
Then the function
\begin{equation*}
u_{f_\epsilon}^{(s)}(x):=
c(n,s)\int_{B_{\epsilon}((\epsilon+1)e)} \frac{(1-\vert x\vert^2)^s}{\vert y-x\vert^n}\,dy
\end{equation*}
is $s$-harmonic in $B_1$, as a consequence of Proposition \ref{MGSS}. Therefore, if we fix $x=-re$ for $r\in (0,1)$, we have that 
\begin{equation*}
\frac{u_{f_\epsilon}^{(s)}(0)}{u_{f_\epsilon}^{(s)}(-re)}=\frac{\displaystyle\int_{B_{\epsilon}((\epsilon+1)e)}\frac{dy}{\vert y\vert^n}}{\displaystyle\int_{B_{\epsilon}((\epsilon+1)e)}\frac{(1- r^2)^s}{\vert y+re\vert^n}\,dy},
\end{equation*}
and thus, by Lebesgue Differentiation Theorem, we conclude that 
\begin{equation*}
\lim_{\epsilon\to 0}\frac{u_{f_\epsilon}^{(s)}(0)}{u_{f_\epsilon}^{(s)}(-re)}=\frac{(1+r)^n}{(1-r^2)^s}.
\end{equation*}
This proves that the constant on the left hand side inequality in \eqref{FracHarn} is optimal. 

Similarly, 
taking~$x=re$, one sees that
\begin{equation*}
\lim_{\epsilon\to 0}\frac{u_{f_\epsilon}^{(s)}(re)}{u_{f_\epsilon}^{(s)}(0)}=\frac{(1-r^2)^s}{(1-r)^n},
\end{equation*}
which shows that the constant on the right hand side inequality in \eqref{FracHarn} is also optimal.
This concludes the proof of Theorem~\ref{JHRGNA:THM}.
\end{proof}

\begin{appendix}

\section{Appendices}

\subsection{Averaging the Fractional Poisson Kernel}\label{SELF}

Here we give a direct proof of the identity pointed out in
footnote~\ref{INTUK} on page~\pageref{INTUK}.
Namely, we establish that, if~$x\in B_1$ and~$\rho>1$,
\begin{equation}\label{PoKeedE} {\mathcal{E}}(x,\rho)=\rho^{n-1}\fint_{\partial B_\rho} P(x,y)\,dH^{n-1}_y,\end{equation}
being~$P$ the Fractional Poisson Kernel in~\eqref{PoKe}.

The identity in~\eqref{PoKeedE} has its own interest and it can be
deduced from our Fractional Malmheden Theorem,
by taking a datum~$f$ concentrating along a given sphere~$\partial B_\rho$.
For the sake of completeness however, we provide here an independent proof,
only based on elementary computations and standard integral formulas.

More specifically, we aim at showing that
\begin{equation}\label{PoKeedE2} 
\fint_{\partial B_1} \frac{dH^{n-1}_\omega}{|x-\rho\omega|^n}
=\frac{\rho^{2-n}}{\rho^2-|x|^2}.
\end{equation}
Indeed, once~\eqref{PoKeedE2}  is established, we deduce from it, \eqref{Edefin999} and~\eqref{PoKe} that
\begin{eqnarray*}
&& \fint_{\partial B_\rho} P(x,y)\,dH^{n-1}_y=
c(n,s)\;(1-\vert x\vert^2)^s \fint_{\partial B_\rho} 
\frac{dH^{n-1}_y}{(\vert y\vert^2-1)^s\;\vert x-y\vert^n}
\\&&\qquad=\frac{c(n,s)\;(1-\vert x\vert^2)^s}{|\partial B_1|\,\rho^{n-1}} \int_{\partial B_1} 
\frac{\rho^{n-1}\,dH^{n-1}_\omega}{(\rho^2-1)^s\;| x-\rho\omega|^n}
\\&&\qquad=\frac{c(n,s)\;(1-\vert x\vert^2)^s}{(\rho^2-1)^s} \fint_{\partial B_1} 
\frac{dH^{n-1}_\omega}{| x-\rho\omega|^n}
=
c(n,s)\,\frac{ \rho^{2-n}\,(1-\vert x\vert^2)^s}{(\rho^2-1)^s
\,(\rho^2-\vert x\vert^2)}\\&&\qquad=\rho^{1-n}{\mathcal{E}}(x,\rho)
\end{eqnarray*}
and this would complete the proof of~\eqref{PoKeedE}.

Hence, we focus now on proving~\eqref{PoKeedE2}. To this end,
we use spherical coordinates on~$\partial B_1$, with~$\theta$, $\theta_1$, $\dots$, $\theta_{n-3}\in[0,\pi]$
and~$\theta_{n-2}\in[0,2\pi)$, see e.g. equation~(A.23) in~\cite{1}
or pages~60-61 in~\cite{MR0185399}, which correspond to~$\omega=\omega(\theta)$ of the form
$$ \begin{cases}
\omega_1=\sin \theta \,\sin \theta_1 \dots \sin \theta_{n-3} \sin \theta_{n-2},\\
\omega_2 =\sin \theta \,\sin \theta_1\dots\sin \theta_{n-3} \cos \theta_{n-2},\\
\omega_3=\sin \theta \,\sin \theta_1 \dots\cos \theta_{n-3},\\
\vdots\\
\omega_n= \cos \theta
\end{cases}$$
and produce a surface element of the form
$$\sin^{n-2}\theta\;
\sin^{n-3}\theta_1\;\dots\;\sin^2\theta_{n-4}\;\sin\theta_{n-3}\,d\theta\,d\theta_1\,\dots\,d\theta_{n-2}.$$
Also, up to a rotation, to prove~\eqref{PoKeedE2} we can assume that~$x=(0,\dots,0,|x|)$.
In this way, we find that
$$|x-\rho\omega|^2=|x|^2+\rho^2-2\rho x\cdot\omega=
|x|^2+\rho^2-2\rho |x|\,\omega_n=|x|^2+\rho^2-2\rho |x|\,\cos\theta,$$
whence
it follows that
\begin{eqnarray*}&&
\int_{\partial B_1} \frac{dH^{n-1}_\omega}{|x-\rho\omega|^n}=
\int_{{\theta,\theta_1,\dots,\theta_{n-3}\in[0,\pi]}\atop{
\theta_{n-2}\in[0,2\pi)}}\frac{\sin^{n-2}\theta\;
\sin^{n-3}\theta_1\;\dots\;\sin^2\theta_{n-4}\;\sin\theta_{n-3}\,d\theta\,d\theta_1\,\dots\,d\theta_{n-2}}{
\big(|x|^2+\rho^2-2\rho |x|\,\cos\theta\big)^{\frac{n}2}
}\\&&\qquad=2\pi\left( \prod_{j=1}^{n-3} \int_0^{\pi} \sin^{n-j-2}\theta_j\;d\theta_j\right)
\int_{\theta \in[0,\pi]}\frac{\sin^{n-2}\theta\,d\theta}{
\big(|x|^2+\rho^2-2\rho |x|\,\cos\theta\big)^{\frac{n}2}
}.
\end{eqnarray*}
Thus, we use the notation~$\tau:=\frac{\rho}{|x|}\in(1,+\infty)$ and, by Proposition~A.9
in~\cite{1}, we deduce that
\begin{eqnarray*}
\int_{\partial B_1} \frac{dH^{n-1}_\omega}{|x-\rho\omega|^n}&=&
\frac{2\pi}{|x|^n}\left( \prod_{k=1}^{n-3} \int_0^{\pi} \sin^{k}\vartheta\,d\vartheta\right)
\int_{0}^\pi\frac{\sin^{n-2}\theta\,d\theta}{
\big(\tau^2+1-2\tau\,\cos\theta\big)^{\frac{n}2}
}\\ &=&
\frac{2\pi}{|x|^n\,\tau^{n-2}(\tau^2-1)}\left( \prod_{k=1}^{n-3} \int_0^{\pi} \sin^{k}\vartheta\,d\vartheta\right)
\int_{0}^\pi \sin^{n-2}\alpha\,d\alpha\\ &=&
\frac{2\pi}{|x|^n\,\tau^{n-2}(\tau^2-1)}\left( \prod_{k=1}^{n-2} \int_0^{\pi} \sin^{k}\vartheta\,d\vartheta\right).
\end{eqnarray*}
This and Proposition~A.10
in~\cite{1} yield that
\begin{eqnarray*}
&&\int_{\partial B_1} \frac{dH^{n-1}_\omega}{|x-\rho\omega|^n}=
\frac{2\pi^{\frac{n}2}}{|x|^n\,\tau^{n-2}(\tau^2-1)\;\Gamma\displaystyle\left(\frac{n}2\right)}=
\frac{|\partial B_1|}{|x|^n\,\tau^{n-2}(\tau^2-1)}=\frac{|\partial B_1|\;\rho^{2-n}}{\rho^2-|x|^2}.
\end{eqnarray*}
The proof of~\eqref{PoKeedE2} is thereby complete.

\subsection{Proof of \eqref{Append}}\label{po2}

We recall~\eqref{Edefin999} and we employ the change of variable~$e:=\omega/|\omega|$ to see that
\begin{equation}\label{ma}
\begin{split}
&\frac{1}{c(n,s)(1-\vert x\vert^2)^s}\int_{1}^\infty\left(
\int_{\partial B_1}
\mathcal{E}(x,\rho)\mathcal{L}_{f_k,e,\rho}(x)\,dH_{e}^{n-1}\right)\,d\rho \\
=\;&\int_{1}^\infty \frac{\rho}{(\rho^2-1)^s(\rho^2-\vert x\vert^2) }
\left(\int_{\partial B_1}\mathcal{L}_{f_k,e,\rho}(x)\,dH_{e}^{n-1}\right)
\,d\rho\\
=\;&\int_{1}^\infty \frac{\rho^2}{\rho^n (\rho^2-1)^s(
\rho^2-\vert x\vert^2)}\left(
\int_{\partial B_\rho}\mathcal{L}_{f_k,\omega/\vert\omega\vert,\vert 
\omega\vert}(x)\,dH_{\omega}^{n-1}\right)\,d\rho\\
=\;&
\int_{\mathbb{R}^n\setminus B_{1}} \frac{\vert y\vert^2}{\vert y\vert^n(\vert y\vert^2-1)^s(\vert y\vert^2-\vert x\vert^2)}\mathcal{L}_{f_k,y/\vert y\vert,\vert y\vert}(x)\,dy.
\end{split}
\end{equation}
It also follows from~\eqref{EFK} that, for a.e. $y\in \mathbb{R}^n\setminus B_1$,
\begin{equation}\label{Mi}
 \mathcal{L}_{f_k,y/\vert y\vert,\vert y\vert}(x)\to \mathcal{L}_{f,y/\vert y\vert,\vert y\vert}(x) \quad
 {\mbox{ as }}k\to+\infty.
\end{equation}
Now we take~$R_0\in(1,R)$ and we deduce from~\eqref{ma} that
\begin{equation}\begin{split}\label{doweu34imsnsdlif}
&\frac{1}{c(n,s)(1-\vert x\vert^2)^s}\int_{1}^\infty\left( \int_{\partial B_1}
\mathcal{E}(x,\rho)\mathcal{L}_{f_k,e,\rho}(x)\,dH_{e}^{n-1}\right)\,d\rho\\=\;&
\int_{B_{R_0}\setminus B_{1}} \frac{\vert y\vert^2}{
\vert y\vert^n(\vert y\vert^2-1)^s(\vert y\vert^2-\vert x\vert^2)}
\mathcal{L}_{f_k,y/\vert y\vert,\vert y\vert}(x)\,dy
\\&\qquad
+\int_{\mathbb{R}^n\setminus B_{R_0}} \frac{\vert y\vert^2}{\vert y\vert^n(\vert y\vert^2-1)^s(\vert y\vert^2-\vert x\vert^2)}
\mathcal{L}_{f_k,y/\vert y\vert,\vert y\vert}(x)\,dy.
\end{split}\end{equation}
Recalling the computation in~\eqref{swireyeghreu50978659650},
for $k$ large enough we have that
\begin{equation*}
 \Vert\mathcal{L}_{f_k,y/\vert y\vert,\vert y\vert}(x)\Vert_{L^\infty(B_{R_0}\setminus B_1)}
 \leq \frac{2R_0}{\sqrt{1-|x|^2}}
\Vert  f \Vert_{L^\infty(B_{R}\setminus B_1)}.
\end{equation*}
Consequently, using this, \eqref{Mi} and the Dominated Convergence
Theorem,
\begin{equation}\label{Mo}\begin{split}&
\lim_{k\to+\infty}\int_{B_{R_0}\setminus B_{1}} \frac{\vert y\vert^2}{
\vert y\vert^n(\vert y\vert^2-1)^s(\vert y\vert^2-\vert x\vert^2)}
\mathcal{L}_{f_k,y/\vert y\vert,\vert y\vert}(x)\,dy\\&\qquad\qquad
= \int_{B_{R_0}\setminus B_{1}} \frac{\vert y\vert^2}{
\vert y\vert^n(\vert y\vert^2-1)^s(\vert y\vert^2-\vert x\vert^2)}
\mathcal{L}_{f,y/\vert y\vert,\vert y\vert}(x)\,dy.
\end{split}\end{equation}

Also, we claim that there exists a subsequence $\lbrace f_{k_j}\rbrace_j$ such that 
\begin{equation}\label{LS111}
\Vert \mathcal{L}_{f_{k_j},y/\vert y\vert,\vert y\vert}(x)
-\mathcal{L}_{f,y/\vert y\vert,\vert y\vert}(x)\Vert_{L_s^1
(\mathbb{R}^n\setminus B_{R_0})}\to 0 \quad{\mbox{ as }}j\to+\infty.
\end{equation}
To show \eqref{LS111}, we recall~\eqref{sowru38tu348t43ty48y458tyhghefweigheo}
and we observe that, for every~$x\in B_1$
and~$y\in\mathbb{R}^n\setminus B_{R_0}$,
\begin{equation}\label{KW}\begin{split}&
\left|\mathcal{L}_{f_k,y/\vert y\vert,\vert y\vert}(x)\right|\\
=\;&
\left| \frac{\left(\frac{x}{|y|}-Q^{x/|y|}_-\left(\frac{y}{|y|}\right)\right)
\cdot \frac{y}{|y|}
}{|Q^{x/|y|}_+\left(\frac{y}{|y|}\right)-Q^{x/|y|}_-\left(
\frac{y}{|y|}\right)|}
\,f_k\left(|y|Q^{x/|y|}_+
\left(\frac{y}{|y|}\right)\right)\right. \\&\qquad\qquad\left.+\frac{\left(Q^{x/|y|}_+
\left(\frac{y}{|y|}\right)
-\frac{x}{|y|}\right)
\cdot \frac{y}{|y|}}{|Q^{x/|y|}_+\left(\frac{y}{|y|}\right)
-Q^{x/|y|}_-\left(\frac{y}{|y|}\right)|}
\,f_k\left(Q^{x/|y|}_-\left(\frac{y}{|y|}\right)\right) \right|\\
\\ \leq\;&
\frac{R_0}{\sqrt{1-\vert x\vert^2}}\bigg[\left|
 f_k\left(\vert y\vert Q_{-}^{x/\vert y\vert}\left(\frac{y}{\vert y\vert}\right)
\right)\right| +\left| f_k\left(\vert y\vert Q_{-}^{x/\vert y\vert}\left(
\frac{y}{\vert
 y\vert}\right)\right)\right|\bigg].
\end{split}\end{equation} 
Moreover, by~\eqref{EFK} there exists a subsequence
$\lbrace f_{k_j}\rbrace_j$ and a
function~$h\in L_s^1(\mathbb{R}^n\setminus B_{1})$
such that $\vert f_{k_j}(y)\vert\leq h(y)$ for a.e. $y\in \mathbb{R}^n\setminus B_1$ (see for instance Theorem~4.9 in~\cite{4}).
Therefore, using this information into \eqref{KW}, we have
\begin{equation}\label{KW!}
\left|\mathcal{L}_{f_{k_j},y/\vert y\vert,\vert y\vert}(x)\right|   \leq
\frac{R_0}{\sqrt{1-\vert x\vert^2}}
\bigg[ h\left(\vert y\vert Q_{-}^{x/\vert y\vert}\left(\frac{y}{\vert y\vert}\right)
\right)+h\left(\vert y\vert Q_{+}^{x/\vert y\vert}\left(\frac{y}{\vert y\vert}\right)
\right)\bigg]
\end{equation} 
for a.e. $y\in \mathbb{R}^n\setminus B_{R_0}$. 

Now we recall the map~$F_\pm$ defined in~\eqref{mappaF}, which
is $C^1$ and invertible, and therefore,
by Theorem~2 in Section~3.3.3 of~\cite{3} and~\eqref{Royhargrove},
we get that
\begin{eqnarray*}
&&\int_{\mathbb{R}^n\setminus B_{R_0}} 
\frac{h\left(\vert y\vert Q_{\pm}^{x/\vert y\vert}(y/\vert y\vert)\right)}{
\left|   y \right|^{n+2s}}
\,dy=
\int_{\mathbb{R}^n\setminus B_{R_0}} 
\frac{h\left(\vert y\vert Q_{\pm}^{x/\vert y\vert}(y/\vert y\vert)\right)}{
\left| \vert y\vert Q_{\pm}^{x/\vert y\vert}(y/\vert y\vert)\right|^{n+2s}}
\,dy\\&&\qquad\qquad=\int_{\mathbb{R}^n\setminus B_{R_0}} \frac{h(z)}{
\vert z\vert^{n+2s}}\vert\det(D F_{\pm}^{-1}(z))\vert \,dz
\le C\int_{\mathbb{R}^n\setminus B_{R_0}} \frac{h(z)}{
\vert z\vert^{n+2s}}  \,dz
.
\end{eqnarray*}
Accordingly, we deduce that
$$h\left(\vert y\vert Q_{\pm}^{x/\vert y\vert}(
y/\vert y\vert)\right)\in L_s^1(\mathbb{R}^n\setminus B_{R_0}).$$
This, the bound in~\eqref{KW!}, the
pointwise convergence in~\eqref{Mi}
and the Dominated Convergence Theorem lead to~\eqref{LS111}, as desired.

Hence, putting together~\eqref{doweu34imsnsdlif}, \eqref{Mo} and~\eqref{LS111},
we obtain that
\begin{equation*}
u_f(x)=\int_{1}^\infty \int_{\partial B_1}\mathcal{E}(x,\rho)\mathcal{L}_{f,e,\rho}(x)\,dH_{e}^{n-1}\, d\rho
\end{equation*}
for each $x\in B_1$, which completes the proof of~\eqref{Append}.
\hfill$\Box$

\subsection{Proof of \eqref{APPEND2}}\label{po3}

The proof of \eqref{APPEND2} is similar to the one of \eqref{Append}. We provide
here the details for the convenience of the reader.

{F}rom~\eqref{Edefin999} with~$n=2$ and the
change of variable~$e:=\omega/|\omega|$,
\begin{equation}\label{swie23uri3poiuytrlkjhgfvbnjkvfgdszfxdghfkeg}
\begin{split}
&\frac{1}{c(2,s)(1-\vert x\vert^2)^s}\int_{1}^\infty\left(
\int_{\partial B_1} \mathcal{E}(x,\rho)f_k(\rho Q^{x/\rho}(e))
\,dH_{e}^{n-1}\right)\,d\rho\\
=\;&\int_{1}^\infty\frac{\rho}{(\rho^2-\vert x\vert^2) (\rho^2-1)^s}
\left(\int_{\partial B_1}f_k(\rho Q^{x/\rho}(e))\,dH_e^{n-1}\right)\,d\rho\\
=\;&\int_{1}^\infty \frac{1}{(\rho^2-\vert x\vert^2) (\rho^2-1)^s}\left(
\int_{\partial B_\rho}f_k(\vert \omega\vert Q^{x/\vert \omega\vert}
(\omega/\vert \omega\vert))\,dH_\omega^{n-1}\right)\,d\rho\\
=\;&\int_{\mathbb{R}^n\setminus B_1}\frac{1}{(\vert y\vert^2-\vert x\vert^2)(\vert y\vert^2-1)^s}f_k(\vert y\vert Q^{x/\vert y\vert} (y/\vert y\vert))\,dy
\end{split}
\end{equation}
By~\eqref{EFK}, we have that, for a.e. $y\in \mathbb{R}^n\setminus B_1$, 
\begin{equation}\label{swie23uri3poiuytrlkjhgfvbnjkvfgdszfxdghfkegBIS}
f_k(\vert y\vert Q^{x/\vert y\vert}(y/\vert y\vert))\to 
f(\vert y\vert Q^{x/\vert y\vert}(y/\vert y\vert)) \quad {\mbox{ as }}
k\to+\infty.
\end{equation} 
Now we take $R_0\in(1,R)$ and we get
from~\eqref{swie23uri3poiuytrlkjhgfvbnjkvfgdszfxdghfkeg} that
\begin{equation}\label{swie23uri3poiuytrlkjhgfvbnjkvfgdszfxdghfkeg22}
\begin{split}
&\frac{1}{c(2,s)(1-\vert x\vert^2)^s}\int_{1}^\infty\left(
\int_{\partial B_1} \mathcal{E}(x,\rho)f_k(\rho Q^{x/\rho}(e))
\,dH_{e}^{n-1}\right)\,d\rho\\
=\;&\int_{B_{R_0}\setminus B_1}\frac{1}{
(\vert y\vert^2-\vert x\vert^2)(\vert y\vert^2-1)^s}
f_k(\vert y\vert Q^{x/\vert y\vert} (y/\vert y\vert))\,dy\\&\qquad\qquad+
\int_{\mathbb{R}^n\setminus B_{R_0}}\frac{1}{
(\vert y\vert^2-\vert x\vert^2)(\vert y\vert^2-1)^s}
f_k(\vert y\vert Q^{x/\vert y\vert} (y/\vert y\vert))\,dy
\end{split}
\end{equation}

Notice that, for $k$ large enough, 
\begin{equation*}
\Vert f_k\Vert_{L^\infty(B_{R_0}\setminus B_1)}
\leq \Vert f\Vert_{L^\infty(B_R\setminus B_1)}. 
\end{equation*} 
As a consequence,
\begin{equation}\begin{split}
\label{swie23uri3poiuytrlkjhgfvbnjkvfgdszfxdghfkeg2233}
&\lim_{k\to+\infty}
\int_{B_{R_0}\setminus B_1}\frac{1}{
(\vert y\vert^2-\vert x\vert^2)(\vert y\vert^2-1)^s}
f_k(\vert y\vert Q^{x/\vert y\vert} (y/\vert y\vert))\,dy\\&\qquad\qquad
=\int_{B_{R_0}\setminus B_1}\frac{1}{
(\vert y\vert^2-\vert x\vert^2)(\vert y\vert^2-1)^s}
f(\vert y\vert Q^{x/\vert y\vert} (y/\vert y\vert))\,dy
\end{split}\end{equation}

Furthemore, recalling~\eqref{EFK}
(see also Theorem~4.9 in~\cite{4})
we deduce the existence of a subsequence $\lbrace f_{k_j}\rbrace_j$
and of a function $h\in L_s^1(\mathbb{R}^n\setminus B_{1})$ such that 
\begin{equation}\label{swiorh4utggxlkjhgfdszxcvshdfuerteroiUYTRd}
\left|f_{k_j}(\vert y\vert Q^{x/\vert y\vert}(y\vert y\vert))\right|\leq h(\vert y\vert Q^{x/\vert y\vert}(y/\vert y\vert))
\end{equation} 
for a.e. $y\in \mathbb{R}^n\setminus B_{R_0}$.

Furthermore, we claim that
\begin{equation}\label{sqwe33lghtrJJJJJJJ978}
{\mbox{$h(\vert y\vert Q^{x/\vert y\vert}(y/\vert y\vert))$ belongs
to $L_s^1(\mathbb{R}^n\setminus B_{R_0})$}}.\end{equation}
Indeed, the function 
\begin{equation*}
\begin{split}
F:\mathbb{R}^n\setminus & B_{R_0}\to \mathbb{R}^n
\setminus B_{R_0}\\
& y\mapsto \vert y\vert Q^{x/\vert y\vert}(y/\vert y\vert)
\end{split}
\end{equation*}
is $C^1$ and invertible. Moreover,
since
$$\lim_{\vert y\vert \to \infty}Q^{x/\vert y\vert}(y/\vert y\vert)=
id_{\partial B_1},$$ 
we find that 
\begin{equation*}
\Vert \det(DF^{-1}(z))\Vert_{L^\infty(\mathbb{R}^n\setminus B_{R_0})}\leq \tilde{C},
\end{equation*}
for some positive constant $\tilde{C}>0$.

{F}rom this and Theorem~2 in Section~3.3.3 of~\cite{3} we have that 
\begin{equation*} \begin{split}&
\int_{\mathbb{R}^n\setminus B_{R_0}}
\frac{h(\vert y\vert Q^{x/\vert y\vert}(y/\vert y\vert ))}{|y|^{2+2s}}\, dy=
\int_{\mathbb{R}^n\setminus B_{R_0}}
\frac{h(\vert y\vert Q^{x/\vert y\vert}(y/\vert y\vert ))}{
|\vert y\vert Q^{x/\vert y\vert}(y/\vert y\vert )|^{2+2s}}\, dy\\&\qquad\qquad
=\int_{\mathbb{R}^n\setminus B_{R_0}} \frac{h(z)}{
\vert z\vert^{{2+2s}}}\vert\det(DF^{-1}(z))\vert \,dz\le
\tilde{C}\int_{\mathbb{R}^n\setminus B_{R_0}} \frac{h(z)}{
\vert z\vert^{{2+2s}}} \,dz,
\end{split}\end{equation*}
which establishes~\eqref{sqwe33lghtrJJJJJJJ978}.

{F}rom~\eqref{swie23uri3poiuytrlkjhgfvbnjkvfgdszfxdghfkegBIS},
\eqref{swiorh4utggxlkjhgfdszxcvshdfuerteroiUYTRd}
and~\eqref{sqwe33lghtrJJJJJJJ978} and the Dominated Convergence
Theorem, we deduce that
$$ 
\Vert f_{k_j}(\vert y\vert Q^{x/\vert y\vert}(y\vert y\vert))
-f(\vert y\vert Q^{x/\vert y\vert}(y\vert y\vert))\Vert_{L_s^1
(\mathbb{R}^n\setminus B_{R_0})}\to 0 \quad{\mbox{ as }}j\to+\infty.
$$

Gathering together this,
\eqref{swie23uri3poiuytrlkjhgfvbnjkvfgdszfxdghfkeg22}
and~\eqref{swie23uri3poiuytrlkjhgfvbnjkvfgdszfxdghfkeg2233},
we conclude that
$$
u_f^{(s)}(x)=\int_{1}^\infty\left(\int_{\partial B_1}\mathcal{E}(x,\rho)
f(\rho Q^{x/\rho}(e))\,dH_e^{n-1}\right)\,d\rho.
$$
This finishes the proof
of~\eqref{APPEND2}.\hfill$\Box$

\end{appendix}

\end{document}